\documentclass[smallextended,referee,envcountsect]{svjour3}
\usepackage [latin1]{inputenc}
\usepackage{amsmath,amssymb}
\usepackage{marvosym}
\usepackage{graphicx}
\usepackage{subcaption}
\usepackage{lipsum}
\usepackage{epstopdf}
\usepackage[numbers,sort&compress]{natbib}
\usepackage[colorlinks,linkcolor=blue,urlcolor=blue,citecolor=blue]{hyperref}

\usepackage{algorithm}
\usepackage{algorithmic}
\usepackage{float}

\usepackage{xcolor}

\smartqed
\usepackage{graphicx}
\journalname{}

\usepackage{fancyhdr}
\usepackage{booktabs}

\usepackage{ntheorem}
\theoremheaderfont{\bfseries\upshape}
\theorembodyfont{\upshape}
\renewtheorem{remark}{\it Remark}[section]
\renewtheorem{example}{Example}[section]
\renewtheorem{property}{Property}[section]
\renewtheorem{proposition}{Proposition}[section]
\renewtheorem{lemma}{Lemma}[section]

\newtheorem*{theorem*}{Theorem}

\begin{document}

\title{Inertial Primal-Dual Dynamics Methods Featuring Implicit Hessian-Driven	Damping for Convex Optimization Problems in Continuous and Discrete Time}

\author{Xiangkai Sun$^{1,2}$ \and Zeying Gao$^{1,2}$ \and Liang He$^{3}$ \and Kok Lay Teo$^{4}$}

\institute{Xiangkai Sun  (\Letter) \at{\small sunxk@ctbu.edu.cn } \\
	\\ Zeying Gao \at {\small gaozeying666@163.com}\\
	\\ Liang He \at{\small liangheee@126.com}\\
	\\ Kok Lay Teo \at{\small K.L.Teo@curtin.edu.cn}\\\\
               $^{1}$School of Mathematics and Statistics,
 Chongqing Technology and Business University,
Chongqing 400067, China\\
\\
$^{2}$Chongqing Key Laboratory of Statistical Intelligent Computing and Monitoring,
 Chongqing Technology and Business University,
Chongqing 400067, China\\
\\
$^{3}$Department of Mathematics, Sichuan University, Chengdu 610065, China\\
\\
$^{4}$School of Mathematical Sciences, Sunway University, 47500 Bandar Sunway, Selangor, Malaysia
}

\date{Received:   / Accepted:  }

\maketitle

\begin{abstract}
 This paper investigates inertial primal-dual dynamics with implicit Hessian-driven damping for strongly convex optimization problems with linear equality constraints. We first establish fast convergence rates for the objective function value error, the feasibility measure, and both the trajectory and its corresponding velocity vector. By appropriately adjusting parameters, we show that the proposed system achieves exponential convergence rates. Through implicit time discretization of the dynamical system, we derive an inertial accelerated primal-dual algorithm for solving the strongly convex optimization problems. Using Lyapunov-based method, we show that the proposed algorithm achieves convergence rates consistent with those of its continuous-time counterpart. We also extend the obtained results to non-smooth convex optimization case. Furthermore, we conduct numerical experiments to illustrate theoretical results.
\end{abstract}
\keywords{Primal-dual dynamics \and Convergence rate \and Inertial accelerated algorithm \and Convex optimization}

\subclass{90C25 \and 37N40 \and 34D05}

\section{Introduction}
Let $\mathbb{R}^n$ be an $n$-dimensional Euclidean space and let $f: \mathbb{R}^n \rightarrow \mathbb{R}$ be a differentiable convex function. Consider the following convex optimization problem with linear equality constraints:
\begin{equation}\label{prob}
\left\{ \begin{array}{cc}
		\min\limits_{x}&{f(x)} \\
        s.t.& Ax=b,
	\end{array}
	\right.
\end{equation}
where $A: \mathbb{R}^{n} \rightarrow \mathbb{R}^{m}$ is a linear operator and $b\in \mathbb{R}^{m}$. As a fundamental model, Problem \eqref{prob} has been used in various fields, such as image recovery \cite{image}, machine learning \cite{machine}, and transportation \cite{trans}.

Inertial primal-dual dynamics method \cite{zeng2023, bot2021}, as a powerful approach for solving Problem \eqref{prob}, has attracted considerable attention from several different perspectives. One of the main reasons is that appropriate time discretization of the dynamical system enables the design of new numerical algorithms with superior performance, thereby providing valuable insights for algorithm design. In order to solve Problem \eqref{prob}, Zeng et al. \cite{zeng2023} first proposed the following primal-dual dynamical system with vanishing damping:
\begin{align} \label{zeng}
	\left\{ \begin{array}{l}
		\ddot{x}(t)+ \frac{\alpha}{t} \dot{x}(t) + \nabla_x {\mathcal{L}}_\rho \big(x(t), \lambda(t) + \theta t \dot{\lambda}(t) \big)=0,\\
		\ddot{\lambda}(t)+ \frac{\alpha}{t} \dot{\lambda}(t) - \nabla_\lambda {\mathcal{L}}_\rho \big(x(t)+ \theta t \dot{x}(t), \lambda(t)\big)=0,
	\end{array}
	\right.
\end{align}
where $\alpha > 0$, $\theta = \max \left\{ \frac{1}{2}, \frac{3}{2\alpha} \right\}$, and $\mathcal{L}_\rho: \mathbb{R}^{n} \times \mathbb{R}^{m}\rightarrow \mathbb{R}$  is the augmented Lagrangian function of Problem \eqref{prob}. When $\alpha > 3$ and $\theta = \frac{1}{2}$, they proved that the primal-dual gap and the feasibility measure of System \eqref{zeng} is $\mathcal{O}(\frac{1}{t^2})$ and $\mathcal{O}(\frac{1}{t})$, respectively. Subsequently, by applying a time discretization to System \eqref{zeng}, Bo\c{t} et al. \cite{bot2023} proposed a fast augmented Lagrangian algorithm for solving Problem \eqref{prob} and established an $\mathcal{O}(\frac{1}{k^2})$ convergence rate for the primal-dual gap, the objective function value error, and the feasibility measure. By incorporating a time-scaling coefficient $\beta(t)$ into System \eqref{zeng}, Hulett and Nguyen \cite{hul2023} improved the convergence rates from $\mathcal{O}(\frac{1}{t^2})$ to $\mathcal{O}(\frac{1}{t^2 \beta(t)})$. Replacing the asymptotic vanishing damping with fixed viscous damping and incorporating a time-scaling coefficient into System \eqref{zeng}, Ding et al. \cite{ding} developed an inertial algorithm for solving Problem \eqref{prob} via an implicit time discretization of this system. They also obtained exponential convergence rates for the primal-dual gap, the objective function value error, and the feasibility measure, without requiring the strong convexity assumption. Moreover, to solve separable convex optimization problems, He et al. \cite{sepa2021} proposed a primal-dual dynamical system with time-dependent positive damping terms and established convergence rates for the proposed system under different choices of the damping coefficients. When the objective function of Problem \eqref{prob} has a ``smooth + nonsmooth" composite structure, Zhang et al. \cite{ZH} proposed a new inertial accelerated primal-dual algorithm via the time discretization of a second-order differential system. For further research on System \eqref{zeng} and its generalizations from both continuous and discrete time perspectives, see \cite{admm, he2022, zhu2025, heo, heo2} and the references therein.

Recently, to reduce the numerical and time discretization complexity inherent in second-order dynamical systems, and to improve computational efficiency, He et al. \cite{HHF2022} proposed the following ``second-order primal" + ``first-order dual" dynamical system:
\begin{eqnarray} \label{2+1}
	\begin{split}
		\left\{ \begin{array}{l}
			\ddot{x}(t)+ \alpha\dot{x}(t) + \gamma(t) \nabla_x {\mathcal{L}}_\rho \big(x(t), \lambda(t)\big)=0,\\
			\dot{\lambda}(t) - \gamma(t) \nabla_\lambda {\mathcal{L}}_\rho \big(x(t) + \theta \dot{x}(t), \lambda(t)\big)=0,
		\end{array}
		\right.
	\end{split}
\end{eqnarray}
where $\alpha,~\theta > 0$ and $\gamma: [t_0,+\infty) \rightarrow (0,+\infty)$ is a time-scaling function. It should be noted that, in System \eqref{2+1}, the inertia term is constructed only for the primal variable. When $\lim \limits_{t \to +\infty} \gamma(t) = +\infty$, they obtained that the asymptotic convergence rate of the primal-dual gap and the feasibility measure along the trajectory generated by System \eqref{2+1} are $\mathcal{O}(\frac{1}{\gamma(t)})$ and $\mathcal{O}\big(\frac{1}{\sqrt{\gamma(t)}}\big)$, respectively. By replacing $\alpha$ with $\frac{\alpha}{t}$ and using Nesterov's acceleration technique, He et al. \cite{he2022l} proposed a fast primal-dual algorithm with nonergodic convergence rates. Under suitable parameter conditions, they proved that the objective function value error and the feasibility measure of the proposed algorithm achieve a rate of $\mathcal{O}(\frac{1}{k^{\alpha-1}})$. For Problem \eqref{prob} with a separable structure, Sun et al. \cite{ZLJ} introduced a Tikhonov regularized ``second-order primal" + ``first-order dual" dynamical system. Using Lyapunov-based analysis, they proved that the fast convergence rates of the system and the strong convergence of the generated trajectory to the minimal-norm solution of Problem \eqref{prob}. By incorporating Tikhonov regularization technique into Nesterov-type primal-dual algorithms, Zhu et al. \cite{zhul} proposed a fast primal-dual algorithm with Tikhonov regularization for solving Problem \eqref{prob}. When the Tikhonov regularization coefficient tends slowly to zero, they established the strong convergence of the primal sequence to the minimal-norm solution of Problem \eqref{prob}. In particular, when the objective function $f$ of Problem \eqref{prob} is a composite convex function, He et al. \cite{he2026} proposed an accelerated primal-dual algorithm based on implicit-explicit time discretization of the corresponding ``second-order primal" + ``first-order dual" differential inclusion system. Under suitable assumptions, they obtained the convergence rates for the primal-dual gap, the objective function value error, and the feasibility measure of the proposed algorithm. For further convergence results on ``second-order primal" + ``first-order dual" dynamical systems and related discrete algorithms, see \cite{mix, hex2025, bat2025, zhu2026}.

Notably, oscillation phenomena frequently occur during the iteration of the considered dynamical systems and their discrete algorithms. To address this issue, in the context of unconstrained optimization problems, many scholars have considered introducing dynamical systems that incorporate explicit Hessian-driven damping terms \cite{alvarez2002second, attouch2016hes, bot2021hes, firstorder, jano, heavy}. However, when using standard explicit or implicit discretization of the dynamical system with explicit Hessian-driven damping, Nesterov-type inertial algorithms cannot be naturally derived. To overcome this limitation, inertial dynamical systems with implicit Hessian-driven damping terms have attracted growing interest across various disciplines, see, e.g. \cite{alecsa2021extension, attouch2021effect, c25, gzy}. Compared to unconstrained optimization problems, there appear to be only a few papers in the literature devoted to the study of inertial primal-dual dynamical systems with Hessian-driven damping for Problem \eqref{prob}. More precisely, He et al. \cite{hehes} proposed the following ``second-order primal" + ``first-order dual" dynamical system with explicit Hessian-driven damping:
\begin{eqnarray} \label{pdhes}
	\left\{ \begin{array}{l}
		\ddot{x}(t)+ \frac{\alpha}{t} \dot{x}(t) + \beta(t) \frac{d}{dt}{\nabla_x \mathcal{L}\big(x(t), \lambda(t)\big)} + \gamma(t)\nabla_x \mathcal{L} \big(x(t), \lambda(t)\big)=0,\\
		\dot{\lambda}(t) - \eta(t) \nabla_\lambda \mathcal{L} \big( x(t)+\frac{t}{\alpha-1} \dot{x}(t), \lambda(t) \big) = 0,
	\end{array}
	\right.
\end{eqnarray}
where $\alpha > 1$, $\beta,
\gamma, \eta: [t_0,+\infty) \rightarrow (0,+\infty)$ are time-scaling functions, and $\mathcal{L}: \mathbb{R}^{n} \times \mathbb{R}^{m}\rightarrow \mathbb{R}$ is the Lagrangian function of Problem \eqref{prob}. They showed that the fast convergence rates of the primal-dual gap, the objective function value error, and the feasibility measure along the trajectory of System \eqref{pdhes} are of order $\mathcal{O}(\frac{1}{t\eta(t)})$. By applying an implicit time discretization of System \eqref{pdhes}, they derived the corresponding algorithm and proved that the proposed algorithm achieves convergence rates that match those of the continuous-time counterpart. Csetnek and L\'{a}szl\'{o} \cite{cehes} studied a second-order primal-dual dynamical system with explicit Hessian-driven damping and Tikhonov regularization, establishing strong convergence of its trajectories to the minimal-norm solution of Problem \eqref{prob}. Sun et al. \cite{hl} introduced a dynamical system equipped with explicit Hessian-driven damping and Tikhonov regularization to solving convex-concave bilinear saddle-point problems. Li et al. \cite{lhl} introduced a second-order dynamical system with implicit Hessian-driven damping. They investigated the convergence properties of this system and demonstrated that this system can effectively reduce oscillations and offer broader applicability, as it does not require the objective function to be twice differentiable.

From what was mentioned above, for solving Problem \eqref{prob}, we see that primal-dual dynamical systems equipped with implicit Hessian-driven damping have garnered far less research attention than their counterparts with explicit Hessian-driven damping. In fact, compared with its explicit counterpart, the implicit Hessian-driven damping avoids the explicit computation and storage of the Hessian matrix, while still effectively reducing oscillations. This makes it particularly attractive for large-scale optimization problems where second-order information is costly or unavailable. This is the main motivation for investigating dynamical systems with implicit Hessian-driven damping to solve Problem \eqref{prob} in this paper.

Motivated by the works in \cite{he2026, lhl}, this paper investigates inertial primal-dual dynamics with implicit Hessian-driven damping for solving Problem \eqref{prob}, where the objective function $f$ is a $\mu$-strongly convex with constant $\mu > 0$ and $\nabla f$ is $L$-Lipschitz continuous with constant $L>0$. To this end, we first introduce a ``second-order primal" + ``first-order dual" dynamical system with implicit Hessian-driven damping as follows:
\begin{align}\label{system}
	\left\{
	\begin{array}{ll}
		\ddot{x}(t) + 2\sqrt{\mu}\dot{x}(t) + \nabla_x \mathcal{L}\big(z(t), \lambda(t)\big)=0, \\ [2mm]
		\dot{\lambda}(t) - \eta(t)\nabla_\lambda \mathcal{L}\big(z(t) + \frac{1}{\sqrt{\mu}}\dot{z}(t), \lambda(t) \big) = 0,
	\end{array}
	\right.
\end{align}
where $z(t) = x(t)+\beta \dot{x}(t)$, $\beta \geq 0$, and $2\sqrt{\mu}$ is a constant viscous damping coefficient, $\frac{1}{\sqrt{\mu}}$ is a constant extrapolation coefficient, $\mathcal{L}: \mathbb{R}^{n} \times \mathbb{R}^{m}\rightarrow \mathbb{R}$ is the Lagrangian function of Problem \eqref{prob} (see \eqref{lglr} for details), $\eta: [t_0,+\infty) \rightarrow (0,+\infty)$ is a non-decreasing differentiable time-scaling function and $\lim \limits_{t \to +\infty} \eta(t) = +\infty$. The term $\nabla_x \mathcal{L}(z(t),\lambda(t))$ in \eqref{system} is referred to as implicit Hessian-driven damping. This is justified by applying the Taylor expansion to $\nabla_x \mathcal{L}(\cdot,\lambda(t))$, which yields
\begin{align*}
	\nabla_x \mathcal{L}(z(t),\lambda(t)) \approx \nabla_x \mathcal{L}(x(t),\lambda(t))+\beta \frac{d}{dt} \nabla_x \mathcal{L}(x(t),\lambda(t)).
\end{align*}
Subsequently, we establish several new asymptotic properties of System \eqref{system}. Moreover, by employing an implicit time discretization of System \eqref{system}, we construct a corresponding inertial accelerated algorithm for Problem \eqref{prob}, under the assumption that the objective function is strongly convex. We also provide a convergence analysis for this inertial accelerated algorithm. The main contributions of this paper are summarized as follows:
\begin{description}
	\item[\textbullet] \textbf{The continuous time level:}
	\textup{We propose a novel ``second-order primal" + ``first-order dual" dynamical system with implicit Hessian-driven damping to solve Problem \eqref{prob}. Compared with explicit Hessian-driven damping \cite{hehes, cehes}, the implicit formulation eliminates the need to store and compute second-order gradient information. This significantly reduces computational cost and broadens applicability, while fully preserving the oscillation-reduction effect. Using the $\mu$-strong convexity of $f$, we establish an $\mathcal{O}\left(\frac{1}{\sqrt{\eta(t)}}\right)$ convergence rate for the objective function value error, the feasibility measure, and both the trajectory and its corresponding velocity vector. We also demonstrate that our method achieves exponential convergence when the time-scaling factor grows exponentially. Moreover, we show that our approach can naturally be extended to the investigation of non-smooth convex optimization problems, extending the convergence rate results from unconstrained optimization problems \cite{firstorder, attouch2021effect} to linearly equality constrained ones.}
	
	\item[\textbullet] \textbf{The discrete time level:} \textup{We develop an inertial primal-dual algorithm via the implicit time discretization of System \eqref{system}. Few, if any, accelerated algorithms have been proposed from the perspectives of discreting primal-dual dynamical systems with implicit Hessian-driven damping, although some studies have explored algorithms derived from systems with explicit Hessian-driven damping or without damping \cite{bot2023, ding, hehes}. To the best of our knowledge, this is the first algorithm derived from a primal-dual dynamical system with implicit Hessian-driven damping. Under some mild conditions on the parameters, we show that our algorithm achieves fast convergence properties that match those of the continuous-time dynamics, thereby overcoming the common difficulty that an improper discretization scheme may lead to a loss of convergence.}
\end{description}

The rest of the paper is organized as follows. In Section 2, we recall some basic notions and present several preliminary results. In Section 3, we first investigate the asymptotic properties of System \eqref{system}. Then, we extend these methods to solve non-smooth convex optimization problems. In Section 4, we first introduce a primal-dual algorithm derived from time discretization of Syetem \eqref{system}. Then, we analyze the convergence behaviors of the proposed algorithm. In Section 5, we present numerical experiments to illustrate theoretical results.

\section{Preliminaries}

Throughout this paper, let $\mathbb{R}^n$ be an $n$-dimensional Euclidean space with inner product $\langle \cdot,\cdot \rangle$ and norm $\lVert \cdot \rVert$. The notation ${L}_{loc}^{1}([ t_0,+\infty))$ denotes the family of locally integrable functions. Let $\psi: \mathbb{R}^{n} \rightarrow \mathbb{R}$ be a continuous differentiable function. For any $\mu>0$, we say that $\psi$ is a $\mu $-strongly convex function iff $ \psi -\frac{\mu}{2}\|\cdot\|$ is a convex function. We say that $\psi$ is gradient Lipschitz continuous (or simply $L$-smooth) iff there exists a constant $L > 0$ such that $\|\nabla \psi(x) - \nabla \psi(y)\| \leq L\|x-y\|$, for all $x, y \in \mathbb{R}^{n}$. Clearly, for the $L$-smooth and $\mu$-strongly convex function $\psi$, it holds that
\begin{equation}\label{L}
	\frac{\mu}{2}\|y-x\|^2 \leq \psi(y) - \psi(x) - \langle \nabla \psi(x), y - x \rangle \leq \frac{L}{2} \|y - x\|^{2},~\forall x, y \in \mathbb{R}^{n}.
\end{equation}
The Lagrangian function associated with Problem  \eqref{prob} is defined as:
\begin{align}\label{lglr}
	\mathcal{L} (x, \lambda) = f(x) + \langle \lambda, Ax-b \rangle.
\end{align}
Let $\Omega$ denote the saddle point set of Problem \eqref{prob}. Then, $(x^*, \lambda^*)\in \Omega$ if and only if
\begin{eqnarray} \label{kkt}
	\left\{ \begin{array}{ll}
		\nabla_x \mathcal{L}( x^*,\lambda^*) = \nabla f(x^*) + A^\top\lambda^* = 0,\\
		\nabla_\lambda \mathcal{L}( x^*,\lambda^*) = Ax^* - b = 0.
	\end{array}
	\right.
\end{eqnarray}

In what follows, we assume that $\Omega \neq\emptyset$. Then, the solution set $S$ of Problem \eqref{prob} is nonempty. It is well-known that $x^*\in S$ if and only if there exists a solution $\lambda^*$ to Problem \eqref{prob} such that $(x^*, \lambda^*)\in\Omega$. This means that
$(x^*, \lambda^*) \in\Omega $ if and only if
\begin{equation*}
	\mathcal{L}(x^*, \lambda) \leqslant \mathcal{L}( x^*, \lambda^*) \leqslant \mathcal{L}(x, \lambda^*),~\forall (x, \lambda) \in \mathbb{R}^{n} \times \mathbb{R}^{m}.
\end{equation*}

\begin{lemma}\cite[Lemma 6]{he2022l} \label{yinli}
	Suppose that $t_0 > 0$, $g: [t_0, +\infty) \to \mathbb{R}^n$ is a continuous function, $h: [t_0, +\infty) \to [0, +\infty)$ is a continuous function, and $C \geq 0$. Let
	\begin{equation*}
		\left\| g(t) + \int_{t_0}^t h(s)g(s) \, ds \right\| \leq C, ~\forall t \geq t_0.
	\end{equation*}
	Then, $\sup_{t \geq t_0} \| g(t) \| \leq 2C$.
\end{lemma}

\begin{lemma}\cite[Lemma 4]{he2022l} \label{yinli2}
	Let $\{r_{k}\}_{k \geq 1}$ be a sequence in $\mathbb{R}^n$, $\{b_{k}\}_{k \geq 1}$ be a sequence in $[0,1)$, and $C \geq 0$. Suppose that
	\[
	\left\| r_{k+1} + \sum_{i=1}^k b_i r_i \right\| \le C, ~\forall k \ge 1.
	\]
	Then, $\underset{k \geq 1}{\sup} \|r_k\| < \|r_1\| + 2C$.
\end{lemma}

\begin{lemma} \label{muml} \cite[Proposition 1]{firstorder}
	Suppose that $f: \mathbb{R}^{n} \rightarrow \mathbb{R} \cup \{+\infty\}$ is a proper lower semi-continuous convex function. Then, for any $\gamma > 0$ and $\mu > 0$, it holds that
	\begin{equation*}
		f \text{ is } \mu \text{-strongly convex} \implies f_\gamma \text{ is strongly convex with modulus } \frac{\mu}{1 + \gamma \mu}.
	\end{equation*}
\end{lemma}

In the sequel, we establish the existence and uniqueness theorem for global strong solutions to System \eqref{system}.

\begin{theorem} \label{unique}
	Suppose that $\nabla f$ is $L$-Lipschitz continuous on $\mathbb{R}^n$ with $L>0$. Let $\eta (t)\in {L}_{loc}^{1}([ t_0,+\infty))$. Then, for initial conditions $(x(t_0),\dot{x}(t_0), \lambda(t_0)) \in \mathbb{R}^n \times \mathbb{R}^n \times \mathbb{R}^m$, System \eqref{system} admits a unique global strong solution.
\end{theorem}

For the detailed proof of Theorem \ref{unique}, we refer the reader to the Appendix.

\section{The Continuous-Time Dynamics Approach}

In this section, by using Lyapunov-based analysis, we investigate the asymptotic properties of System \eqref{system}. Specifically, we establish fast convergence rates for the objective function value error, the feasibility measure, and both the trajectory and its corresponding velocity vector.

\begin{proposition} \label{prop}
	Let $(x,\lambda ):[{t}_{0},+\infty )\rightarrow \mathbb{R}^{n}\times \mathbb{R}^{m}$ be a solution of System \eqref{system} and let $({x}^{*}, {\lambda }^{*})\in \Omega $. The function $W: [t_0, +\infty) \rightarrow [0, +\infty)$ is defined as follows:
	\begin{equation*}
		\begin{split}
			W(t) := \mathcal{L}(x(t)+\beta\dot{x}(t), \lambda^*)-\mathcal{L}(x^*, \lambda^*) + \frac{1}{2}\|\sqrt{\mu}(x(t)-x^*)+\dot{x}(t)\|^2 + \frac{\mu}{4}\|x(t)-x^{*}\|^2.
		\end{split}
	\end{equation*}
	Then, for any $t \geq t_0$, it holds that
	\begin{equation} \label{hinequ}
		\begin{split}
			\dot{W}(t) + \frac{\sqrt{\mu}}{2} W(t) \leq &~ -\frac{\beta}{2} \|\nabla_x \mathcal{L}(x(t)+\beta\dot{x}(t), \lambda(t))\|^2 - \frac{\sqrt{\mu}}{4}(3- 2\beta\sqrt{\mu} - 3\beta^2\mu) \|\dot{x}(t)\|^2 \\
			& - \frac{\sqrt{\mu}}{\eta(t)} \langle \lambda(t)-\lambda^{*}, \dot{\lambda}(t) \rangle.
		\end{split}
	\end{equation}
\end{proposition}

\begin{proof}
	Now, we first analyze the time derivative of $W(t)$. Clearly, using $\nabla_x \mathcal{L}(x(t)+\beta\dot{x}(t), \lambda^*)=\nabla_x \mathcal{L}(x(t)+\beta\dot{x}(t), \lambda(t))-A^\top(\lambda(t)-\lambda^*)$, we obtain
	\begin{equation}\label{h1}
		\begin{split}
			\dot{W}(t)
			= &~ \langle \nabla_x \mathcal{L}(x(t)+\beta\dot{x}(t), \lambda^*), \dot{x}(t)+ \beta\ddot{x}(t) \rangle \\
			& + \langle \sqrt{\mu}(x(t)-x^*)+\dot{x}(t), \sqrt{\mu}\dot{x}(t) + \ddot{x}(t) \rangle + \frac{\mu}{2} \langle x(t)-x^*, \dot{x}(t) \rangle \\
			= &~ \langle \nabla_x \mathcal{L}(x(t)+\beta\dot{x}(t), \lambda(t)), (1-2\beta\sqrt{\mu})\dot{x}(t)-\beta\nabla_x \mathcal{L}(x(t)+\beta\dot{x}(t), \lambda(t)) \rangle \\
			& - \langle A^\top (\lambda(t)-\lambda^*), (1-2\beta\sqrt{\mu})\dot{x}(t)-\beta\nabla_x \mathcal{L}(x(t)+\beta\dot{x}(t), \lambda(t)) \rangle \\
			& + \langle \sqrt{\mu}(x(t)-x^*)+\dot{x}(t), -\sqrt{\mu}\dot{x}(t) - \nabla_x \mathcal{L}(x(t)+\beta\dot{x}(t), \lambda(t)) \rangle \\
			& + \frac{\mu}{2} \langle x(t)-x^*, \dot{x}(t) \rangle \\
			= &~ -2\beta\sqrt{\mu}\langle \nabla_x \mathcal{L}(x(t)+\beta\dot{x}(t), \lambda(t)), \dot{x}(t) \rangle - \beta\|\nabla_x \mathcal{L}(x(t)+\beta\dot{x}(t), \lambda(t))\|^2 \\
			& - (1-2\beta\sqrt{\mu}) \langle \lambda(t)-\lambda^*, A\dot{x}(t) \rangle + \beta \langle A^\top (\lambda(t)-\lambda^*), \nabla_x \mathcal{L}(x(t)+\beta\dot{x}(t), \lambda(t)) \rangle \\
			& - \frac{\mu}{2} \langle x(t) - x^*, \dot{x}(t) \rangle - \sqrt{\mu} \|\dot{x}(t)\|^2 - \sqrt{\mu} \langle \nabla_x \mathcal{L}(x(t)+\beta\dot{x}(t), \lambda^*), x(t)- x^* \rangle \\
			& - \sqrt{\mu} \langle A^\top (\lambda(t)-\lambda^*), x(t)- x^* \rangle.
		\end{split}
	\end{equation}
	From the $\mu$-strong convexity of $f$, we deduce that
	\begin{equation*}
		\begin{split}
			& - \langle \nabla_x \mathcal{L}(x(t)+\beta\dot{x}(t), \lambda^*), x(t)+\beta\dot{x}(t) - x^* \rangle \\
			= &~ - \langle \nabla f(x(t)+\beta\dot{x}(t)), x(t)+\beta\dot{x}(t) - x^* \rangle - \langle A^\top \lambda^*, x(t)+\beta\dot{x}(t) - x^* \rangle \\
			\leq &~ f(x^*) - f(x(t)+\beta\dot{x}(t)) - \frac{\mu}{2}\|x(t)+\beta\dot{x}(t) - x^*\|^2 \\
			& - \langle A^\top \lambda^*, x(t)+\beta\dot{x}(t) - x^* \rangle \\
			= &~ \mathcal{L}(x^*, \lambda^*) - \mathcal{L}(x(t)+\beta\dot{x}(t), \lambda^*) - \frac{\mu}{2}\|x(t)+\beta\dot{x}(t) - x^*\|^2.
		\end{split}
	\end{equation*}
	Then,
	\begin{equation*}
		\begin{split}
			& - \langle \nabla_x \mathcal{L}(x(t)+\beta\dot{x}(t), \lambda^*), x(t) - x^* \rangle \\
			= &~ - \langle \nabla_x \mathcal{L}(x(t)+\beta\dot{x}(t), \lambda^*), x(t)+\beta\dot{x}(t) - x^* \rangle + \beta \langle \nabla_x \mathcal{L}(x(t)+\beta\dot{x}(t), \lambda^*), \dot{x}(t) \rangle \\
			\leq &~ \mathcal{L}(x^*, \lambda^*) - \mathcal{L}(x(t)+\beta\dot{x}(t), \lambda^*) - \frac{\mu}{2}\|x(t)+\beta\dot{x}(t) - x^*\|^2 \\
			& + \beta \langle \nabla_x \mathcal{L}(x(t)+\beta\dot{x}(t), \lambda(t)), \dot{x}(t) \rangle - \beta \langle \lambda(t)-\lambda^*, A\dot{x}(t) \rangle.
		\end{split}
	\end{equation*}
	This together with \eqref{h1} yields
	\begin{equation*}
		\begin{split}
			\dot{W}(t)
			\leq &~ -\sqrt{\mu} \big(\mathcal{L}(x(t)+\beta\dot{x}(t), \lambda^*) - \mathcal{L}(x^*, \lambda^*)\big) \\
			& - \beta\sqrt{\mu} \langle \nabla_x \mathcal{L}(x(t)+\beta\dot{x}(t), \lambda(t)), \dot{x}(t) \rangle - \beta\|\nabla_x \mathcal{L}(x(t)+\beta\dot{x}(t), \lambda(t))\|^2 \\
			& - \frac{\mu\sqrt{\mu}}{2}\|x(t)+\beta\dot{x}(t) - x^*\|^2 - (1-\beta\sqrt{\mu}) \langle \lambda(t)-\lambda^*, A\dot{x}(t) \rangle \\
			& + \beta \langle A^\top (\lambda(t)-\lambda^*), \nabla_x \mathcal{L}(x(t)+\beta\dot{x}(t), \lambda(t)) \rangle - \frac{\mu}{2} \langle x(t) - x^*, \dot{x}(t) \rangle \\
			& - \sqrt{\mu} \|\dot{x}(t)\|^2 - \sqrt{\mu} \langle A^\top (\lambda(t)-\lambda^*), x(t)- x^* \rangle.
		\end{split}
	\end{equation*}
	Note that
	\begin{equation*}
		\begin{split}
			\frac{\sqrt{\mu}}{2}W(t) = &~ \frac{\sqrt{\mu}}{2}\big(\mathcal{L}(x(t)+\beta\dot{x}(t), \lambda^*)-\mathcal{L}(x^*, \lambda^*)\big) + \frac{3\mu\sqrt{\mu}}{8}\|x(t)-x^*\|^2 \\
			& + \frac{\sqrt{\mu}}{4}\|\dot{x}(t)\|^2 + \frac{\mu}{2}\langle x(t) - x^*, \dot{x}(t) \rangle.
		\end{split}
	\end{equation*}
	Then,
	\begin{equation}\label{*}
		\begin{split}
			& \dot{W}(t) + \frac{\sqrt{\mu}}{2} W(t) \\
			\leq &~ -\frac{\sqrt{\mu}}{2}\big(\mathcal{L}(x(t)+\beta\dot{x}(t), \lambda^*) - \mathcal{L}(x^*, \lambda^*)\big) - \beta\sqrt{\mu} \langle \nabla_x \mathcal{L}(x(t)+\beta\dot{x}(t), \lambda(t)), \dot{x}(t) \rangle \\
			& - \beta\|\nabla_x \mathcal{L}(x(t)+\beta\dot{x}(t), \lambda(t))\|^2 - \frac{\mu\sqrt{\mu}}{2}\|x(t)+\beta\dot{x}(t) - x^*\|^2 \\
			& - (1-\beta \sqrt{\mu}) \langle \lambda(t)-\lambda^*, A\dot{x}(t) \rangle + \beta \langle A^\top (\lambda(t)-\lambda^*), \nabla_x \mathcal{L}(x(t)+\beta\dot{x}(t), \lambda(t)) \rangle \\
			& - \frac{3\sqrt{\mu}}{4} \|\dot{x}(t)\|^2 - \sqrt{\mu} \langle A^\top (\lambda(t)-\lambda^*), x(t)- x^* \rangle + \frac{3\mu\sqrt{\mu}}{8}\|x(t)-x^*\|^2.
		\end{split}
	\end{equation}
	Combining the $\mu$-strong convexity of $f$ with \eqref{kkt} yields
	\begin{align} \label{fstr}
		\begin{split}
			& \mathcal{L}(x(t)+\beta\dot{x}(t), \lambda^*)-\mathcal{L}(x^*, \lambda^*) \\
			\geq & ~ \langle \nabla f(x^*), x(t)+\beta\dot{x}(t)-x^{*} \rangle + \frac{\mu}{2}\|x(t)+\beta\dot{x}(t)-x^*\|^2 + \langle \lambda^{*}, A(x(t)+\beta\dot{x}(t))-b \rangle \\
			= & ~  \frac{\mu}{2}\|x(t)+\beta\dot{x}(t)-x^*\|^2.
		\end{split}
	\end{align}
	Therefore, it follows from \eqref{*} and \eqref{fstr} that
	\begin{equation*}
		\begin{split}
			&~\dot{W}(t) + \frac{\sqrt{\mu}}{2} W(t)\\
			\leq &~ - \beta\sqrt{\mu} \langle \nabla_x \mathcal{L}(x(t)+\beta\dot{x}(t), \lambda(t)), \dot{x}(t) \rangle - \beta\|\nabla_x \mathcal{L}(x(t)+\beta\dot{x}(t), \lambda(t))\|^2 \\
			& - \frac{3\mu\sqrt{\mu}}{4}\|x(t)+\beta\dot{x}(t) - x^*\|^2 - (1-\beta \sqrt{\mu}) \langle \lambda(t)-\lambda^*, A\dot{x}(t) \rangle \\
			& + \beta \langle A^\top (\lambda(t)-\lambda^*), \nabla_x \mathcal{L}(x(t)+\beta\dot{x}(t), \lambda(t)) \rangle - \frac{3\sqrt{\mu}}{4} \|\dot{x}(t)\|^2 \\
			& - \sqrt{\mu} \langle A^\top (\lambda(t)-\lambda^*), x(t)- x^* \rangle + \frac{3\mu\sqrt{\mu}}{8}\|x(t)-x^*\|^2\\
			\leq &~ - \beta\sqrt{\mu} \langle \nabla_x \mathcal{L}(x(t)+\beta\dot{x}(t), \lambda(t)), \dot{x}(t) \rangle - \beta\|\nabla_x \mathcal{L}(x(t)+\beta\dot{x}(t), \lambda(t))\|^2 \\
			& - (1-\beta \sqrt{\mu}) \langle \lambda(t)-\lambda^*, A\dot{x}(t) \rangle \\
			& + \beta \langle A^\top (\lambda(t)-\lambda^*), \nabla_x \mathcal{L}(x(t)+\beta\dot{x}(t), \lambda(t)) \rangle - \frac{3\sqrt{\mu}}{4}(1-\beta^2 \mu) \|\dot{x}(t)\|^2 \\
			& - \sqrt{\mu} \langle A^\top (\lambda(t)-\lambda^*), x(t)- x^* \rangle,
		\end{split}
	\end{equation*}
	where the last inequality holds because $\|x(t)+\beta\dot{x}(t) - x^*\|^2 \geq \frac{1}{2}\|x(t)-x^*\|^2 - \beta^2 \|\dot{x}(t)\|^2$. Furthermore, note that
	\begin{align*}
		- \beta\sqrt{\mu} \langle \nabla_x \mathcal{L}(x(t)+\beta\dot{x}(t), \lambda(t)), \dot{x}(t) \rangle \leq ~\frac{\beta}{2} \|\nabla_x \mathcal{L}(x(t)+\beta\dot{x}(t), \lambda(t))\|^2 + \frac{\beta \mu}{2}\|\dot{x}(t)\|^2.
	\end{align*}
	Thus,
	\begin{equation*}
		\begin{split}
			&~\dot{W}(t) + \frac{\sqrt{\mu}}{2} W(t)\\
			\leq &~ - \frac{\beta}{2}\|\nabla_x \mathcal{L}(x(t)+\beta\dot{x}(t), \lambda(t))\|^2 - (1-\beta \sqrt{\mu}) \langle \lambda(t)-\lambda^*, A\dot{x}(t) \rangle \\
			& + \beta \langle A^\top (\lambda(t)-\lambda^*), \nabla_x \mathcal{L}(x(t)+\beta\dot{x}(t), \lambda(t)) \rangle \\
			& - \frac{\sqrt{\mu}}{4}(3-2\beta\sqrt{\mu}-3\beta^2 \mu) \|\dot{x}(t)\|^2 - \sqrt{\mu} \langle A^\top (\lambda(t)-\lambda^*), x(t)- x^* \rangle \\
			= &~ -\frac{\beta}{2} \|\nabla_x \mathcal{L}(x(t)+\beta\dot{x}(t), \lambda(t))\|^2 - \frac{\sqrt{\mu}}{4}(3- 2\beta\sqrt{\mu} - 3\beta^2\mu) \|\dot{x}(t)\|^2 \\
			& - \frac{\sqrt{\mu}}{\eta(t)} \langle \lambda(t)-\lambda^{*}, \dot{\lambda}(t) \rangle,
		\end{split}
	\end{equation*}
	where the last equality holds because the second equation of System \eqref{system} and
	\begin{align*}
		 \sqrt{\mu} \nabla_\lambda \mathcal{L}\Big(z(t)+\frac{1}{\sqrt{\mu}}\dot{z}(t),\lambda(t)\Big)
		=   (1-\beta\sqrt{\mu})A\dot{x}(t)-\beta A\nabla_x \mathcal{L}(z(t),\lambda(t))+\sqrt{\mu}A(x(t)-x^*).
	\end{align*}
	The proof is complete. \qed
	
\end{proof}

\begin{theorem} \label{lianxu}
	Let $(x,\lambda): [{t}_{0},+\infty) \rightarrow \mathbb{R}^{n}\times \mathbb{R}^{m}$ be a solution of System \eqref{system}. Suppose that $0 \leq \beta \leq  \frac{\sqrt{10}-1}{3\sqrt{\mu}}$ and $\dot{\eta}(t) \leq \frac{\sqrt{\mu}}{2}\eta(t)$. Then, for any $({x}^{*}, {\lambda }^{*})\in \Omega $, the trajectory $\{(x(t), \lambda(t))\}$ is bounded, and the following results are satisfied:
	\begin{align*}
		\left\{
		\begin{array}{lll}
			|f(x(t)) - f(x^{*})| = \mathcal{O} \left(\frac{1}{\sqrt{\eta(t)}}\right),~\text{as} ~t \rightarrow +\infty,\\ [3mm]
			\|Ax(t)-b\| = \mathcal{O} \left(\frac{1}{\sqrt{\eta(t)}}\right),~\text{as} ~t \rightarrow +\infty,\\ [3mm]
			\|x(t) - x^{*}\| = \mathcal{O} \left(\frac{1}{\sqrt{\eta(t)}}\right),~\text{as} ~t \rightarrow +\infty,\\ [3mm]
			\|\dot{x}(t)\|= \mathcal{O} \left(\frac{1}{\sqrt{\eta(t)}}\right),~\text{as} ~t \rightarrow +\infty.
		\end{array}
		\right.
	\end{align*}
\end{theorem}

\begin{proof}
	We first introduce the energy function $E: [t_0, +\infty) \rightarrow [0, +\infty)$ as follows:
	\begin{align*}
		E(t) := \eta(t) W(t) + \frac{\sqrt{\mu}}{2}\|\lambda(t)-\lambda^{*}\|^2,
	\end{align*}
	where $W(t)$ is defined in Proposition \ref{prop}.
	
	We now analyze the time derivative of $E(t)$. Obviously, it follows from \eqref{hinequ} that
	\begin{equation} \label{muE}
		\begin{split}
			\dot{E}(t) = &~ \dot{\eta}(t) W(t) + \eta(t) \dot{W}(t) + \sqrt{\mu}\langle \lambda(t)-\lambda^{*}, \dot{\lambda}(t)\rangle \\
			\leq &~ \Big(\dot{\eta}(t) - \frac{\sqrt{\mu}}{2} \eta(t)\Big) W(t) - \frac{\beta}{2}\eta(t) \|\nabla_x \mathcal{L}(x(t)+\beta\dot{x}(t), \lambda(t))\|^2 \\
			& - \frac{\sqrt{\mu}}{4}(3- 2\beta \sqrt{\mu}-3\beta^2 \mu)\eta(t)\|\dot{x}(t)\|^2.
		\end{split}
	\end{equation}
	Since $0 \leq \beta \leq  \frac{\sqrt{10}-1}{3\sqrt{\mu}}$, we have $3-2\beta\sqrt{\mu}-3\beta^2 \mu \geq 0$. This together with $\dot{\eta}(t) \leq \frac{\sqrt{\mu}}{2}\eta(t)$ and \eqref{muE} yields $\dot{E}(t) \leq 0$, for all $t \geq t_0$. This means that $$E(t) \leq E(t_0), \forall t \geq t_0.$$ By virtue of $\eta(t)W(t) \leq E(t) \leq E(t_0)$, it is clear that the trajectory $\{(x(t), \lambda(t))\}$ is bounded and the following results are satisfied:
	\begin{align}
		&\mathcal{L}(x(t)+\beta\dot{x}(t), \lambda^*)-\mathcal{L}(x^*, \lambda^*) = \mathcal{O} \left(\frac{1}{\eta(t)}\right), ~\text{as}~t \rightarrow +\infty, \label{1l}\\
		& \|\sqrt{\mu}(x(t)-x^*)+\dot{x}(t)\|= \mathcal{O} \left(\frac{1}{\sqrt{\eta(t)}}\right), ~\text{as}~t \rightarrow +\infty, \label{2xd} \\
		&\|x(t) - x^{*}\| = \mathcal{O} \left(\frac{1}{\sqrt{\eta(t)}}\right), ~\text{as}~t \rightarrow +\infty, \label{3x}
	\end{align}
	and
	\begin{align} \label{lam}
		\|\lambda(t)-\lambda^*\|^2 \leq \frac{2}{\sqrt{\mu}}E(t_0).
	\end{align}
	Clearly, combining $\|\dot{x}(t)\| \leq \|\sqrt{\mu}(x(t)-x^*)+\dot{x}(t)\| + \sqrt{\mu}\|x(t)-x^*\|$ with \eqref{2xd} and \eqref{3x}, it follows that
	\begin{align}\label{xd}
		\|\dot{x}(t)\| = \mathcal{O} \left(\frac{1}{\sqrt{\eta(t)}}\right),~\text{as}~t \rightarrow +\infty.
	\end{align}
	From \eqref{system}, we deduce that
	\begin{eqnarray}
		\begin{split} \label{lamda}
			\lambda(t)-\lambda(t_0)= \int_{t_0}^t \dot{\lambda}(s)ds
			= & ~ \int_{t_0}^t \eta(s)\left(Az(s) - b + \frac{1}{\sqrt{\mu}} A\dot{z}(s)\right)ds \\
			= & ~ \frac{1}{\sqrt{\mu}} \big[\eta(t)(Az(t)-b)-\eta(t_0)(Az(t_0)-b)\big] \\
			& + \frac{1}{\sqrt{\mu}} \int_{t_0}^t \big( \sqrt{\mu} \eta(s) - \dot{\eta}(s)\big)\big(Az(s)-b\big)ds.
		\end{split}
	\end{eqnarray}
	From \eqref{lam}, we obtain $\|\lambda(t)-\lambda(t_0)\| \leq \|\lambda(t)-\lambda^{*}\| + \|\lambda(t_0)-\lambda^{*}\| \leq 2\sup_{t \geq t_0}\|\lambda(t)-\lambda^{*}\| < +\infty.$ Let $g(t) = \eta(t)(Az(t)-b)$ and $h(t) = \frac{\sqrt{\mu} \eta(t)-\dot{\eta}(t)}{\eta(t)}$. Then, it follows from \eqref{lamda} that
	\[
	\left\|g(t)+ \int_{t_0}^t h(s)g(s) \, ds \right\| \leq C,~\forall t \geq t_0,
	\]
	where $C = 2\sqrt{\mu}\sup_{t \geq t_0} \|\lambda(t)-\lambda^*\| + \eta(t_0)\|Az(t_0)-b\| < +\infty$. This together with Lemma \ref{yinli} implies $\sup_{t \geq t_0} \eta(t)\|Az(t)-b\| < 2C$, which means that
	\begin{align}\label{ad}
		\|A(x(t)+\beta\dot{x}(t)) - b\| = \mathcal{O}\left(\frac{1}{\eta(t)}\right),~\text{as}~t \rightarrow +\infty.
	\end{align}
	This together with $\|Ax(t)-b\| \leq \|A(x(t)+\beta\dot{x}(t)) - b\| + \beta \|A\|\|\dot{x}(t)\|$ and \eqref{xd} yields
	\begin{align*}
		\|Ax(t) - b\| = \mathcal{O}\left(\frac{1}{\sqrt{\eta(t)}}\right),~\text{as}~t \rightarrow +\infty.
	\end{align*}
	Note that
	\begin{equation*}
		\begin{split}
			 |f(x(t)+\beta\dot{x}(t)) - f(x^{*})|
			\leq &~ \mathcal{L}(x(t)+\beta\dot{x}(t), \lambda^*)-\mathcal{L}(x^*, \lambda^*) +  |\langle \lambda^{*}, A(x(t)+\beta\dot{x}(t))-b \rangle| \\
			\leq &~ \mathcal{L}(x(t)+\beta\dot{x}(t), \lambda^*)-\mathcal{L}(x^*, \lambda^*) +  \|\lambda^{*}\| \|A(x(t)+\beta\dot{x}(t))-b\|.
		\end{split}
	\end{equation*}
	Then, it follows from \eqref{1l} and \eqref{ad} that
	\begin{align}\label{ff}
		|f(x(t)+\beta\dot{x}(t)) - f(x^{*})| = \mathcal{O} \left(\frac{1}{\eta(t)}\right), ~\text{as}~t \rightarrow +\infty.
	\end{align}
	Taking \eqref{L} into account, it is easy to show that
	\begin{align*}
		f(x(t)) - f(x(t)+\beta\dot{x}(t)) \leq &~ \langle \nabla f(x(t)+\beta\dot{x}(t)), -\beta\dot{x}(t) \rangle + \frac{L\beta^2}{2}\|\dot{x}(t)\|^2 \\
		\leq &~ \beta \|\nabla f(x(t)+\beta\dot{x}(t))\| \|\dot{x}(t)\| + \frac{L\beta^2}{2}\|\dot{x}(t)\|^2
	\end{align*}
	and
	\begin{align*}
		f(x(t)+\beta\dot{x}(t))-f(x(t)) \leq &~ \langle \nabla f(x(t)),\beta\dot{x}(t) \rangle + \frac{L\beta^2}{2}\|\dot{x}(t)\|^2 \\
		\leq &~ \beta\|\nabla f(x(t))\|\|\dot{x}(t)\|+\frac{L\beta^2}{2}\|\dot{x}(t)\|^2.
	\end{align*}
	Combining the Lipschitz continuity of $\nabla f$, \eqref{xd} and the boundedness of $\{x(t)\}$, we obtain
	\begin{align*}
		|f(x(t)) - f(x(t)+\beta\dot{x}(t))| = \mathcal{O} \left(\frac{1}{\sqrt{\eta(t)}}\right),~\text{as}~t \rightarrow +\infty.
	\end{align*}
	This together with \eqref{ff} and $|f(x(t))-f(x^*)| \leq |f(x(t)) - f(x(t)+\beta\dot{x}(t))| + |f(x(t)+\beta\dot{x}(t)) - f(x^*)|$ implies
	\begin{align*}
		|f(x(t))-f(x^*)| = \mathcal{O} \left(\frac{1}{\sqrt{\eta(t)}}\right),~\text{as}~t \rightarrow +\infty.
	\end{align*}
	The proof is complete. \qed
\end{proof}

In the specific case where $\dot{\eta}(t) = \frac{\sqrt{\mu}}{2} \eta(t)$, we obtain $\eta(t) = \eta(t_0)e^{\frac{\sqrt{\mu}}{2}(t-t_0)}$. Then, by applying Theorem \ref{lianxu}, we establish exponential convergence rates for the objective function value error, the feasibility measure, and both the trajectory and its corresponding velocity vector.

\begin{corollary} \label{coro}
	Let $(x,\lambda): [{t}_{0},+\infty) \rightarrow \mathbb{R}^{n}\times \mathbb{R}^{m}$ be a solution of System \eqref{system}. Suppose that $0 \leq \beta \leq \frac{\sqrt{10}- 1}{3\sqrt{\mu}}$ and $\eta(t) = \eta(t_0)e^{\frac{\sqrt{\mu}}{2}(t-t_0)}$. Then, for any $({x}^{*}, {\lambda }^{*})\in \Omega $, the trajectory $\{(x(t), \lambda(t))\}$ is bounded, and the following results are satisfied:
	\begin{align*}
		\left\{
		\begin{array}{lllll}
			|f(x(t)) - f(x^{*})| = \mathcal{O}\left(e^{-\frac{\sqrt{\mu}}{4}t}\right),~\text{as}~t \rightarrow +\infty,\\ [3mm]
			\|Ax(t)-b\| = \mathcal{O}\left(e^{-\frac{\sqrt{\mu}}{4}t}\right),~\text{as}~t \rightarrow +\infty, \\ [3mm]
			\|x(t)- x^{*}\| = \mathcal{O}\left(e^{-\frac{\sqrt{\mu}}{4}t}\right),~\text{as}~t \rightarrow +\infty,\\ [3mm]
			\|\dot{x}(t)\| = \mathcal{O}\left(e^{-\frac{\sqrt{\mu}}{4}t}\right),~\text{as}~t \rightarrow +\infty.
		\end{array}
		\right.
	\end{align*}
\end{corollary}

\begin{remark}
	\textup{ While existing primal-dual dynamical systems with implicit Hessian-driven damping exhibit linear convergence \cite{lhl}, in this paper we demonstrate that an exponential convergence rate can be achieved under strong convexity of the objective function and appropriate assumptions on the parameters.}
\end{remark}

  In the sequel, we show that our approach can naturally be extended to the investigation of non-smooth convex optimization problems. To do this, we consider the convex optimization problem \eqref{prob}, where $f$ is a proper lower semi-continuous convex function. In order to adapt System \eqref{system} to non-smooth convex situation, we proposed the following differential inclusion system:
  \begin{align} \label{system1}
  	\left\{
  	\begin{array}{ll}
  		\ddot{x}(t) + 2\sqrt{\mu}\dot{x}(t) + \partial_x \mathcal{L}\big(z(t), \lambda(t)\big) \ni 0, \\ [2mm]
  		\dot{\lambda}(t) - \eta(t)\nabla_\lambda \mathcal{L}\big(z(t) + \frac{1}{\sqrt{\mu}}\dot{z}(t), \lambda(t) \big) = 0,
  	\end{array}
  	\right.
  \end{align}
  where $z(t) = x(t) + \beta \dot{x}(t)$ and $\partial_x \mathcal{L}$ denotes the subdifferential of $\mathcal{L}$ with respect to $x$.

  Let $f_\gamma$ be the Moreau envelope of $f$ with index $\gamma > 0$, which is defined as:
  \[
  f_\gamma(x) = \min_{y \in \mathbb{R}^n} \left\{ f(y) + \frac{1}{2\gamma} \|y - x\|^2 \right\}, ~ \forall x \in \mathbb{R}^n.
  \]
  Clearly, $f_\gamma$ is a continuously differentiable convex function, and $\nabla f_\gamma$ is $\gamma^{-1}$-Lipschitz continuous. According to Lemma \ref{muml}, in this case, the convergence rates of System \eqref{system1} can be established by investigating the properties of the following system:
	\begin{align*}
		\left\{
		\begin{array}{ll}
			\ddot{x}_\gamma(t) + \frac{2\sqrt{\mu}}{\sqrt{1 + \gamma \mu}}\dot{x}_\gamma(t) + \nabla_x \mathcal{L}_\gamma \big(z_\gamma(t), \lambda_\gamma(t)\big) = 0, \\ [2mm]
			\dot{\lambda}_\gamma(t) -  \eta(t)\nabla_\lambda \mathcal{L}_\gamma \big(z_\gamma(t) + \frac{1}{\sqrt{\mu}}\dot{z}_\gamma(t), \lambda_\gamma(t) \big) =0,
		\end{array}
		\right.
	\end{align*}
	where $z_\gamma = x_\gamma + \beta \dot{x}_\gamma$ and $\mathcal{L}_\gamma (z_\gamma, \lambda_\gamma) = f_\gamma (z_\gamma) + \langle \lambda_\gamma, Ax_\gamma -b \rangle$.
	
	Following the Moreau-Yosida regularization reported in \cite[Section 2]{MY} and using arguments similar to those in Theorem \ref{lianxu} and Corollary \ref{coro}, we obtain the following results as $\gamma \rightarrow 0$.

\begin{theorem} \label{lianxu2}
	Let $(x,\lambda): [{t}_{0},+\infty) \rightarrow \mathbb{R}^{n}\times \mathbb{R}^{m}$ be a solution of System \eqref{system1}. Suppose that $0 \leq \beta \leq \frac{\sqrt{10}-1}{3\sqrt{\mu}}$ and $\dot{\eta}(t) \leq \frac{\sqrt{\mu}}{2}\eta(t)$. Then, for any $({x}^{*}, {\lambda }^{*})\in \Omega $, the trajectory $\{(x(t), \lambda(t))\}$ is bounded, and the following results are satisfied:
	\begin{align*}
		\left\{
		\begin{array}{lllll}
			|f(x(t)) - f(x^{*})| = \mathcal{O} \left(\frac{1}{\sqrt{\eta(t)}}\right),~\text{as}~t \rightarrow +\infty, \\ [3mm]
			\|Ax(t)-b\| = \mathcal{O} \left(\frac{1}{\sqrt{\eta(t)}}\right),~\text{as}~t \rightarrow +\infty,\\ [3mm]
			\|x(t) - x^{*}\| = \mathcal{O} \left(\frac{1}{\sqrt{\eta(t)}}\right),~\text{as}~t \rightarrow +\infty,\\ [3mm]
			\|\dot{x}(t)\| = \mathcal{O} \left(\frac{1}{\sqrt{\eta(t)}}\right),~\text{as}~t \rightarrow +\infty.
		\end{array}
		\right.
	\end{align*}
\end{theorem}

\begin{corollary}
	Let $(x,\lambda): [{t}_{0},+\infty) \rightarrow \mathbb{R}^{n}\times \mathbb{R}^{m}$ be a solution of System \eqref{system1}. Suppose that $0 \leq \beta \leq \frac{\sqrt{10}-1}{3\sqrt{\mu}}$ and $\eta(t) = \eta(t_0)e^{\frac{\sqrt{\mu}}{2}(t-t_0)}$. Then, for any $({x}^{*}, {\lambda }^{*})\in \Omega $, the trajectory $\{(x(t), \lambda(t))\}$ is bounded, and the following results are satisfied:
	\begin{align*}
		\left\{
		\begin{array}{llll}
			|f(x(t)) - f(x^{*})| = \mathcal{O}\left(e^{-\frac{\sqrt{\mu}}{4}t}\right),~\text{as}~t \rightarrow +\infty,\\ [3mm]
			\|Ax(t)-b\| = \mathcal{O}\left(e^{-\frac{\sqrt{\mu}}{4}t}\right),~\text{as}~t \rightarrow +\infty,\\ [3mm]
			\|x(t) - x^{*}\| = \mathcal{O}\left(e^{-\frac{\sqrt{\mu}}{4}t}\right),~\text{as}~t \rightarrow +\infty,\\ [3mm]
			\|\dot{x}(t)\| = \mathcal{O}\left(e^{-\frac{\sqrt{\mu}}{4}t}\right),~\text{as}~t \rightarrow +\infty.
		\end{array}
		\right.
	\end{align*}
\end{corollary}

\section{The Inertial Accelerated Primal-Dual Algorithm}

In this section, via a natural implicit time discretization of System \eqref{system}, we propose an inertial accelerated primal-dual algorithm for solving Problem \eqref{prob}. Under the appropriate assumptions of parameters, we investigate the asymptotic properties of the proposed algorithm.

Firstly, we perform time discretization for System \eqref{system} to obtain the corresponding discrete scheme. Consider the constant stepsize $\sqrt{s}$ and set $t_k=k\sqrt{s}$, $x_k=x(t_k)$, $\lambda_k=\lambda(t_k)$ and $\eta_k=\eta(t_k)$. By implicit time discretization to System \eqref{system}, we obtain
\begin{eqnarray}\label{lis}
	\left\{ \begin{array}{ll}
		& \dfrac{x_{k+1}-2x_{k} + x_{k-1}}{s} + 2\sqrt{\mu}\dfrac{x_{k+1}-x_{k}}{\sqrt{s}} + \nabla_x \mathcal{L}(z_{k+1}, \lambda_{k+1}) = 0, \\ [4mm]
		& \dfrac{\lambda_{k+1}-\lambda_k}{\sqrt{s}} - \eta_k \nabla_\lambda \mathcal{L}\left(z_{k+1}+\dfrac{1}{\sqrt{\mu}}\dfrac{z_{k+1}-z_k}{\sqrt{s}},\lambda_{k+1} \right)= 0,
	\end{array}
	\right.
\end{eqnarray}
where $z_k = x_k + \frac{\beta}{\sqrt{s}}(x_k - x_{k-1})$. Denote $\frac{\beta}{\sqrt{s}}$ and ${\sqrt{s}}\eta_k$ as $\beta$ and $\eta_k$, respectively. Then,  $z_k = x_k + \beta(x_k - x_{k-1})$ and \eqref{lis} becomes
\begin{eqnarray} \label{scheme}
	\left\{ \begin{array}{ll}
		& x_{k+1} = x_{k} + \frac{1}{1+2\sqrt{\mu s}}(x_k - x_{k-1}) - \frac{s}{1+2\sqrt{\mu s}}\nabla_x \mathcal{L}(z_{k+1}, \lambda_{k+1}), \\ [3mm]
		& \lambda_{k+1} = \lambda_{k} + \eta_k \nabla_\lambda \mathcal{L}\Big(z_{k+1}+\frac{1}{\sqrt{\mu s}}(z_{k+1}-z_k),\lambda_{k+1}\Big).
	\end{array}
	\right.
\end{eqnarray}

Based on \eqref{scheme}, we present an inertial accelerated primal-dual algorithm for solving Problem \eqref{prob}.
\begin{algorithm}[H]
	\floatname{algorithm}{Algorithm}
	\caption{Inertial Accelerated Primal-Dual Algorithm}
	\label{kuangjia}
	\textbf{Initialization}: Choose $x_0=x_1=z_1 \in \mathbb{R}^n$, $\lambda_0=\lambda_1 \in \mathbb{R}^m$, $\beta \geq 0$ and $\eta_0=1$.
	
	\ \ \textbf{for} $k = 1,2,\dots$ \textbf{do}
	
	\ \ \ \ $\textbf{Step 1}$: Compute
	\begin{align*}
		\bar{x}_{k} = x_k + \frac{1}{1+2\sqrt{\mu s}}(x_k - x_{k-1})~\text{and}~z_k = x_k + \beta(x_k - x_{k-1}).
	\end{align*}
	
	\ \ \ \ $\textbf{Step 2}$: Choose $\eta_k > 0$ and update the primal variable
	\begin{align*}
		x_{k+1} =\underset{x\in\mathbb{R}^n}{\arg\min} &\left\{
		f\big(x+\beta(x-x_k)\big) + \frac{(1+\beta)(1+2\sqrt{\mu s})}{2s}\left\|x - \bar{x}_k + \frac{s}{1+2\sqrt{\mu s}}A^\top (\lambda_k -\eta_k b) \right\|^2 \right.\\
		&~~~~~~~~~~~~~~~~~~~~~~~~~\left. +\frac{(1+\beta)^2 (1+\sqrt{\mu s})}{2\sqrt{\mu s}}\eta_k\left\|A\left(x-\frac{\beta}{1+\beta}x_k\right) - \frac{1}{(1+\beta)(1+\sqrt{\mu s})}Az_k\right\|^2 \right\}.
	\end{align*}
	
	\ \ \ \ $\textbf{Step 3}$: Compute $z_{k+1} = x_{k+1} + \beta(x_{k+1} - x_k)$ and update the dual variable
	\begin{align*}
		\lambda_{k+1} = \lambda_{k} + \eta_k \left(Az_{k+1} - b + \frac{1}{\sqrt{\mu s}}A(z_{k+1} - z_k) \right).
	\end{align*}
	
	\ \ \ \ \textbf{If} a stopping condition is satisfied \textbf{then}
	
	\ \ \ \ \ \ \textbf{return} $\left( x_{k+1}, \lambda_{k+1} \right)$.
	
	\ \ \ \ \textbf{end}
	
	\ \ \textbf{end for}
\end{algorithm}

The following proposition shows that Algorithm \ref{kuangjia} is equivalent to the time discretization scheme \eqref{scheme}.
\begin{proposition} \label{equ}
	Algorithm \ref{kuangjia} is equivalent to the scheme \eqref{scheme}.	
\end{proposition}
\begin{proof}
	First, it follows from Step 2 of Algorithm \ref{kuangjia} that
	\begin{align*}
		0 = &~ s\nabla f(x_{k+1}+\beta(x_{k+1}-x_k)) + \left((1+2\sqrt{\mu s})I+ \frac{\sqrt{s}(1+\beta)(1+\sqrt{\mu s})}{\sqrt{\mu}}\eta_k A^\top A\right)x_{k+1} \\
		& - (1+2\sqrt{\mu s})\bar{x}_k + sA^\top \lambda_k - s\eta_k A^\top b - \frac{\beta\sqrt{s}(1+\sqrt{\mu s})}{\sqrt{\mu}}\eta_k A^\top Ax_k - \frac{\sqrt{s}}{\sqrt{\mu}}\eta_k A^\top Az_k.
	\end{align*}
	This together with Step 1 of Algorithm \ref{kuangjia} yields
	\begin{align*}
		& (1+2\sqrt{\mu s})(x_{k+1}-x_k) \\
		= &~ x_k - x_{k-1} - s\nabla f(z_{k+1}) \\
		& - sA^\top \left[\lambda_k + \eta_k\left( \frac{(1+\beta)(1+\sqrt{\mu s})}{\sqrt{\mu s}}Ax_{k+1} -\frac{\beta(1+\sqrt{\mu s})}{\sqrt{\mu s}}Ax_k - \frac{1}{\sqrt{\mu s}} Az_k - b\right)\right] \\
		= &~ x_k - x_{k-1} - s\nabla f(z_{k+1}) -sA^\top \left[ \lambda_{k} + \eta_k \left(Az_{k+1} - b + \frac{1}{\sqrt{\mu s}}A(z_{k+1} - z_k) \right) \right] .
	\end{align*}
	Using Step 3 of Algorithm \ref{kuangjia}, it follows that
	\begin{align}\label{@@}
		\begin{split}
			(1+2\sqrt{\mu s})(x_{k+1}-x_k) = &~ x_k - x_{k-1} - s\nabla f(z_{k+1}) -sA^\top \lambda_{k+1} \\
			= & ~ x_k - x_{k-1} - s\nabla_x \mathcal{L} (z_{k+1}, \lambda_{k+1}).
		\end{split}
	\end{align}
	This corresponds to the first equation in \eqref{scheme}.
	
	Clearly, Step 3 of Algorithm \ref{kuangjia} corresponds to the second equation in \eqref{scheme}.
	
	In summary, Algorithm \ref{kuangjia} is equivalent to the scheme described in \eqref{scheme}. This completes the proof. \qed
\end{proof}

In the following, we investigate the fast convergence properties of Algorithm \ref{kuangjia}.
\begin{proposition} \label{prop2}
	Let $\{(x_k,\lambda_k)\}_{k \geq 1}$ be the sequence generated by Algorithm \ref{kuangjia} and let $(x^*, \lambda^*)\in \Omega $. Define the positive sequence $W_k$ as follows:
	\begin{equation} \label{hfunc2}
		\begin{split}
			W_k := \mathcal{L}(z_k, \lambda^*)-\mathcal{L}(x^*, \lambda^*) + \frac{1}{2}\|v_k\|^2 + \frac{\mu}{4}\|x_k-x^*\|^2,
		\end{split}
	\end{equation}
	where $z_k=x_k+\beta(x_k - x_{k-1})$ and $v_k = \sqrt{\mu}(x_k-x^*) + \frac{1}{\sqrt{s}}(x_k - x_{k-1})$. Then, for any $k \geq 1$, it holds that
	\begin{equation} \label{hinequ2}
		\begin{split}
			& W_{k+1} - W_k + \frac{\sqrt{\mu s}}{2} W_{k+1} \\
			\leq &~ -\frac{s}{6}(1+4\beta)\|\nabla_x \mathcal{L}(z_{k+1}, \lambda_{k+1})\|^2 - \frac{3\sqrt{\mu}}{4}\Big(\frac{1}{\sqrt{s}} - \beta\sqrt{\mu} - \beta^2 \mu \sqrt{s}\Big)\|x_{k+1}-x_k\|^2 \\
			& - \frac{\mu}{2}\|z_{k+1}-z_k\|^2 -\frac{\sqrt{\mu s}}{\eta_k} \langle \lambda_{k+1} - \lambda^*, \lambda_{k+1} - \lambda_k \rangle.
		\end{split}
	\end{equation}
\end{proposition}

\begin{proof}
	Clearly, from the first equation of \eqref{scheme}, we obtain
	\begin{align} \label{step2}
		x_{k+1} - 2x_k + x_{k-1} = - 2\sqrt{\mu s}(x_{k+1}-x_k) - s \nabla_x \mathcal{L}(z_{k+1}, \lambda_{k+1}).
	\end{align}
	Together with the definition of $v_k$ and \eqref{step2}, it yields
	\begin{align*}
		\begin{split}
			v_{k+1} - v_k = &~ \sqrt{\mu}(x_{k+1}-x_k) + \frac{1}{\sqrt{s}}(x_{k+1} - 2x_k + x_{k-1}) \\
			= &~ - \sqrt{\mu}(x_{k+1}-x_k) -  \sqrt{s} \nabla_x \mathcal{L}(z_{k+1}, \lambda_{k+1}).
		\end{split}
	\end{align*}
	Thus,
	\begin{align} \label{vpingf}
		\begin{split}
			& \frac{1}{2}\|v_{k+1}\|^2 - \frac{1}{2}\|v_k\|^2 \\
			= &~ \langle v_{k+1} - v_k, v_{k+1} \rangle - \frac{1}{2}\|v_{k+1} - v_k\|^2 \\
			= &~ \left\langle - \sqrt{\mu}(x_{k+1}-x_k) - \sqrt{s} \nabla_x \mathcal{L}(z_{k+1}, \lambda_{k+1}), \sqrt{\mu}(x_{k+1}-x^*) +\frac{1}{\sqrt{s}} (x_{k+1} - x_k) \right\rangle  \\
			& - \frac{1}{2}\|- \sqrt{\mu}(x_{k+1}-x_k) - \sqrt{s}\nabla_x \mathcal{L}(z_{k+1}, \lambda_{k+1})\|^2 \\
			= &~ -\mu \langle x_{k+1}-x_k, x_{k+1}-x^* \rangle - \frac{\sqrt{\mu}}{\sqrt{s}}\|x_{k+1}-x_k\|^2 \\
			& - \sqrt{\mu s} \langle \nabla_x \mathcal{L}(z_{k+1}, \lambda_{k+1}), x_{k+1}-x^* \rangle \\
			& - \langle \nabla_x \mathcal{L}(z_{k+1}, \lambda_{k+1}), x_{k+1}-x_k \rangle - \frac{\mu}{2}\|x_{k+1}-x_k\|^2 \\
			& - \frac{s}{2}\|\nabla_x \mathcal{L}(z_{k+1}, \lambda_{k+1})\|^2  - \sqrt{\mu s} \langle \nabla_x \mathcal{L}(z_{k+1}, \lambda_{k+1}), x_{k+1}-x_k \rangle \\
			= &~ -\mu \langle x_{k+1}-x_k, x_{k+1}-x^* \rangle - \Big( \frac{\sqrt{\mu}}{\sqrt{s}} + \frac{\mu}{2} \Big)\|x_{k+1}-x_k\|^2 \\
			& - \sqrt{\mu s} \langle \nabla_x \mathcal{L}(z_{k+1}, \lambda_{k+1}), x_{k+1}-x^* \rangle  - \frac{s}{2}\|\nabla_x \mathcal{L}(z_{k+1}, \lambda_{k+1})\|^2 \\
			& - (1+\sqrt{\mu s}) \langle \nabla_x \mathcal{L}(z_{k+1}, \lambda_{k+1}), x_{k+1}-x_k \rangle.
		\end{split}
	\end{align}
	By the $\mu$-strong convexity of $f$, it follows from \eqref{L} that for any $x, y \in \mathbb{R}^n$,
	\begin{equation} \label{strl}
		\begin{split}
			\mathcal{L}(y, \lambda^*) - \mathcal{L}(x, \lambda^*)
			= &~ f(y) + \langle \lambda^*, Ay-b \rangle - f(x) - \langle \lambda^*, Ax-b \rangle\\
			\geq &~ \langle \nabla f(x), y-x \rangle + \frac{\mu}{2}\|y-x\|^2 + \langle \lambda^*, A(y-x) \rangle \\
			=  &~ \langle \nabla_x \mathcal{L}(x, \lambda^*), y-x \rangle + \frac{\mu}{2}\|y-x\|^2.
		\end{split}
	\end{equation}
	For \eqref{strl}, set $y=z_k$ and $x=z_{k+1}$. Then,
	\begin{equation}
		\begin{split} \label{llz}
			\mathcal{L}(z_{k+1}, \lambda^*) - \mathcal{L}(z_k, \lambda^*)
			\leq \langle \nabla_x \mathcal{L}(z_{k+1}, \lambda^*), z_{k+1}-z_k \rangle - \frac{\mu}{2}\|z_{k+1}-z_k\|^2.
		\end{split}
	\end{equation}
	Combining \eqref{hfunc2}, \eqref{vpingf} and \eqref{llz}, we have
	\begin{equation} \label{h-h}
		\begin{split}
			W_{k+1} - W_k = &~  \mathcal{L}(z_{k+1}, \lambda^*) - \mathcal{L}(z_k, \lambda^*) + \frac{1}{2}\|v_{k+1}\|^2 - \frac{1}{2}\|v_k\|^2 \\
			& + \frac{\mu}{4}\|x_{k+1}-x^*\|^2 - \frac{\mu}{4}\|x_k-x^*\|^2 \\
			= &~  \mathcal{L}(z_{k+1}, \lambda^*) - \mathcal{L}(z_k, \lambda^*) + \frac{1}{2}\|v_{k+1}\|^2 - \frac{1}{2}\|v_k\|^2 \\
			& + \frac{\mu}{2}\langle x_{k+1}-x_k, x_{k+1}-x^* \rangle - \frac{\mu}{4}\|x_{k+1}-x_k\|^2 \\
			\leq &~ \langle \nabla_x \mathcal{L}(z_{k+1}, \lambda^*), z_{k+1}-z_k \rangle - \frac{\mu}{2}\|z_{k+1}-z_k\|^2 \\
			&~ -\frac{\mu}{2} \langle x_{k+1}-x_k, x_{k+1}-x^* \rangle - \left( \frac{\sqrt{\mu}}{\sqrt{s}} + \frac{3\mu}{4} \right)\|x_{k+1}-x_k\|^2 \\
			& - \sqrt{\mu s} \langle \nabla_x \mathcal{L}(z_{k+1}, \lambda_{k+1}), x_{k+1}-x^* \rangle  - \frac{s}{2}\|\nabla_x \mathcal{L}(z_{k+1}, \lambda_{k+1})\|^2 \\
			& - (1+\sqrt{\mu s}) \langle \nabla_x \mathcal{L}(z_{k+1}, \lambda_{k+1}), x_{k+1}-x_k \rangle.
		\end{split}
	\end{equation}
	From \eqref{step2}, it follows that
	\begin{equation} \label{**}
		\begin{split}
			z_{k+1} - z_k = &~ x_{k+1} - x_k + \beta(x_{k+1} - 2x_k + x_{k-1}) \\
			= &~ (1- 2\beta\sqrt{\mu s})(x_{k+1}-x_k) - \beta s\nabla_x \mathcal{L}(z_{k+1}, \lambda_{k+1}).
		\end{split}
	\end{equation}
	Note that $\nabla_x \mathcal{L}(z_{k+1}, \lambda^*)=\nabla_x \mathcal{L}(z_{k+1}, \lambda_{k+1})-A^\top(\lambda_{k+1}-\lambda^*)$. Then, from \eqref{h-h} and \eqref{**}, we obtain
	\begin{equation} \label{h-h2}
		\begin{split}
			& W_{k+1} - W_k \\
			\leq &~ \langle \nabla_x \mathcal{L}(z_{k+1}, \lambda_{k+1}), (1- 2\beta\sqrt{\mu s})(x_{k+1}-x_k) - \beta s\nabla_x \mathcal{L}(z_{k+1}, \lambda_{k+1}) \rangle \\
			& - \langle A^\top (\lambda_{k+1} - \lambda^*), (1- 2\beta\sqrt{\mu s})(x_{k+1}-x_k) - \beta s\nabla_x \mathcal{L}(z_{k+1}, \lambda_{k+1}) \rangle \\
			& - \frac{\mu}{2}\|z_{k+1}-z_k\|^2 -\frac{\mu}{2} \langle x_{k+1}-x_k, x_{k+1}-x^* \rangle \\
			& - \left( \frac{\sqrt{\mu}}{\sqrt{s}} + \frac{3\mu}{4} \right)\|x_{k+1}-x_k\|^2 - \sqrt{\mu s} \langle \nabla_x \mathcal{L}(z_{k+1}, \lambda^*), x_{k+1}-x^* \rangle \\
			& - \sqrt{\mu s} \langle A^\top (\lambda_{k+1} - \lambda^*), x_{k+1}-x^* \rangle - \frac{s}{2}\|\nabla_x \mathcal{L}(z_{k+1}, \lambda_{k+1})\|^2 \\
			& - (1+\sqrt{\mu s}) \langle \nabla_x \mathcal{L}(z_{k+1}, \lambda_{k+1}), x_{k+1}-x_k \rangle \\
			= &~ -\sqrt{\mu s}(1 + 2\beta) \langle \nabla_x \mathcal{L}(z_{k+1}, \lambda_{k+1}), x_{k+1}-x_k \rangle \\
			& - \Big(\frac{s}{2} + \beta s \Big)\| \nabla_x \mathcal{L}(z_{k+1}, \lambda_{k+1})\|^2 \\
			& - (1- 2\beta\sqrt{\mu s}) \langle A^\top (\lambda_{k+1} - \lambda^*), x_{k+1}-x_k \rangle \\
			& + \beta s\langle A^\top (\lambda_{k+1} - \lambda^*), \nabla_x \mathcal{L}(z_{k+1}, \lambda_{k+1}) \rangle - \frac{\mu}{2}\|z_{k+1}-z_k\|^2 \\
			& -\frac{\mu}{2} \langle x_{k+1}-x_k, x_{k+1}-x^* \rangle - \left( \frac{\sqrt{\mu}}{\sqrt{s}} + \frac{3\mu}{4} \right)\|x_{k+1}-x_k\|^2 \\
			& - \sqrt{\mu s} \langle \nabla_x \mathcal{L}(z_{k+1}, \lambda^*), x_{k+1}-x^* \rangle - \sqrt{\mu s} \langle A^\top (\lambda_{k+1} - \lambda^*), x_{k+1}-x^* \rangle.
		\end{split}
	\end{equation}
	Together with \eqref{strl} and $\nabla_x \mathcal{L}(z_{k+1}, \lambda^*)=\nabla_x \mathcal{L}(z_{k+1}, \lambda_{k+1})-A^\top(\lambda_{k+1}-\lambda^*)$, it gives
	\begin{equation*}
		\begin{split}
			& - \langle \nabla_x \mathcal{L}(z_{k+1}, \lambda^*), x_{k+1}-x^* \rangle \\
			= &~ - \langle \nabla_x \mathcal{L}(z_{k+1}, \lambda^*), z_{k+1}-x^* \rangle + \beta \langle \nabla_x \mathcal{L}(z_{k+1}, \lambda^*), x_{k+1}-x_k \rangle \\
			\leq &~ \mathcal{L}(x^*, \lambda^*) - \mathcal{L}(z_{k+1}, \lambda^*) - \frac{\mu}{2}\|z_{k+1}-x^*\|^2 \\
			& + \beta \langle \nabla_x \mathcal{L}(z_{k+1}, \lambda_{k+1}), x_{k+1}-x_k \rangle - \beta \langle A^\top(\lambda_{k+1} - \lambda^*), x_{k+1}-x_k \rangle.
		\end{split}
	\end{equation*}
	Substituting the above inequality into \eqref{h-h2} yields
	\begin{equation*}
		\begin{split}
			W_{k+1} - W_k
			\leq &~ -\sqrt{\mu s}(1 + \beta) \langle \nabla_x \mathcal{L}(z_{k+1}, \lambda_{k+1}), x_{k+1}-x_k \rangle \\
			& - \Big(\frac{s}{2} + \beta s \Big)\|\nabla_x \mathcal{L}(z_{k+1}, \lambda_{k+1})\|^2 \\
			& - (1- \beta\sqrt{\mu s}) \langle A^\top (\lambda_{k+1} - \lambda^*), x_{k+1}-x_k \rangle \\
			& + \beta s\langle A^\top (\lambda_{k+1} - \lambda^*), \nabla_x \mathcal{L}(z_{k+1}, \lambda_{k+1}) \rangle - \frac{\mu}{2}\|z_{k+1}-z_k\|^2 \\
			& -\frac{\mu}{2} \langle x_{k+1}-x_k, x_{k+1}-x^* \rangle - \left( \frac{\sqrt{\mu}}{\sqrt{s}} + \frac{3\mu}{4} \right)\|x_{k+1}-x_k\|^2 \\
			& - \sqrt{\mu s}\big(\mathcal{L}(z_{k+1}, \lambda^*) - \mathcal{L}(x^*, \lambda^*) \big) - \frac{\mu\sqrt{\mu s}}{2}\|z_{k+1}-x^*\|^2 \\
			& - \sqrt{\mu s} \langle A^\top (\lambda_{k+1} - \lambda^*), x_{k+1}-x^* \rangle.
		\end{split}
	\end{equation*}
	Note that
	\begin{equation*}
		\begin{split}
			\frac{\sqrt{\mu s}}{2} W_{k+1} = &~ \frac{\sqrt{\mu s}}{2} \big(\mathcal{L}(z_{k+1}, \lambda^*)-\mathcal{L}(x^*, \lambda^*)\big) \\
			& + \frac{\sqrt{\mu s}}{4}\left\|\sqrt{\mu}(x_{k+1}-x^*) + \frac{1}{\sqrt{s}} (x_{k+1} - x_k)\right\|^2 +\frac{\mu\sqrt{\mu s}}{8}\|x_{k+1}-x^*\|^2 \\
			= &~ \frac{\sqrt{\mu s}}{2} \big(\mathcal{L}(z_{k+1}, \lambda^*)-\mathcal{L}(x^*, \lambda^*)\big) + \frac{3\mu\sqrt{\mu s}}{8}\|x_{k+1}-x^*\|^2 \\
			& + \frac{\sqrt{\mu}}{4\sqrt{s}}\|x_{k+1} - x_k\|^2 + \frac{\mu}{2}\langle x_{k+1} - x_k, x_{k+1}-x^* \rangle.
		\end{split}
	\end{equation*}
	Thus,
	\begin{equation*}
		\begin{split}
			& W_{k+1} - W_k + \frac{\sqrt{\mu s}}{2} W_{k+1} \\
			\leq &~ -\sqrt{\mu s}(1 + \beta) \langle \nabla_x \mathcal{L}(z_{k+1}, \lambda_{k+1}), x_{k+1}-x_k \rangle - \Big(\frac{s}{2} + \beta s \Big)\| \nabla_x \mathcal{L}(z_{k+1}, \lambda_{k+1})\|^2 \\
			& - (1- \beta\sqrt{\mu s}) \langle A^\top (\lambda_{k+1} - \lambda^*), x_{k+1}-x_k \rangle \\
			& + \beta s \langle A^\top (\lambda_{k+1} - \lambda^*), \nabla_x \mathcal{L}(z_{k+1}, \lambda_{k+1}) \rangle - \frac{\mu}{2}\|z_{k+1}-z_k\|^2 \\
			& - \left( \frac{3\sqrt{\mu}}{4\sqrt{s}} + \frac{3\mu}{4} \right)\|x_{k+1}-x_k\|^2 - \frac{\sqrt{\mu s}}{2}\big(\mathcal{L}(z_{k+1}, \lambda^*) - \mathcal{L}(x^*, \lambda^*) \big) \\
			& - \frac{\mu\sqrt{\mu s}}{2}\|z_{k+1}-x^*\|^2 - \sqrt{\mu s} \langle A^\top (\lambda_{k+1} - \lambda^*), x_{k+1}-x^* \rangle \\
			& + \frac{3\mu\sqrt{\mu s}}{8}\|x_{k+1}-x^*\|^2.
		\end{split}
	\end{equation*}
	Combining \eqref{kkt} with \eqref{strl}, we deduce that
	\begin{align}\label{listr}
		\begin{split}
			\mathcal{L}(z_{k+1}, \lambda^*)-\mathcal{L}(x^*, \lambda^*)
			\geq \frac{\mu}{2}\|z_{k+1}-x^*\|^2.
		\end{split}
	\end{align}
	Therefore,
	\begin{equation} \label{h-h4}
		\begin{split}
			& W_{k+1} - W_k + \frac{\sqrt{\mu s}}{2} W_{k+1} \\
			\leq &~ -\sqrt{\mu s}(1 + \beta) \langle \nabla_x \mathcal{L}(z_{k+1}, \lambda_{k+1}), x_{k+1}-x_k \rangle - \Big(\frac{s}{2} + \beta s \Big)\|\nabla_x \mathcal{L}(z_{k+1}, \lambda_{k+1})\|^2 \\
			& - (1- \beta\sqrt{\mu s}) \langle A^\top (\lambda_{k+1} - \lambda^*), x_{k+1}-x_k \rangle \\
			& + \beta s \langle A^\top (\lambda_{k+1} - \lambda^*), \nabla_x \mathcal{L}(z_{k+1}, \lambda_{k+1}) \rangle - \frac{\mu}{2}\|z_{k+1}-z_k\|^2 \\
			& - \left( \frac{3\sqrt{\mu}}{4\sqrt{s}} + \frac{3\mu}{4} \right)\|x_{k+1}-x_k\|^2 - \frac{3\mu\sqrt{\mu s}}{4}\|z_{k+1}-x^*\|^2 \\
			& - \sqrt{\mu s} \langle A^\top (\lambda_{k+1} - \lambda^*), x_{k+1}-x^* \rangle + \frac{3\mu\sqrt{\mu s}}{8}\|x_{k+1}-x^*\|^2.
		\end{split}
	\end{equation}
	Note that
	\begin{align*}
		\begin{split}
			& -\sqrt{\mu s}(1 + \beta) \langle \nabla_x \mathcal{L}(z_{k+1}, \lambda_{k+1}), x_{k+1}-x_k \rangle \\
			\leq &~\frac{s}{3}(1 + \beta)\|\nabla_x \mathcal{L}(z_{k+1}, \lambda_{k+1})\|^2 + \frac{3\mu}{4}(1 + \beta)\|x_{k+1}-x_k\|^2
		\end{split}
	\end{align*}
	and
	\begin{align*}
		\begin{split}
			\|z_{k+1}-x^*\|^2 = &~ \|x_{k+1}-x^* + \beta(x_{k+1}-x_k)\|^2
			\geq \frac{1}{2}\|x_{k+1}-x^*\|^2 - \beta^2\|x_{k+1}-x_k\|^2.
		\end{split}
	\end{align*}
	Substituting the above inequalities into \eqref{h-h4}, we obtain
	\begin{equation*}
		\begin{split}
			& W_{k+1} - W_k + \frac{\sqrt{\mu s}}{2} W_{k+1} \\
			\leq &~ -\frac{s}{6}(1+4\beta)\|\nabla_x \mathcal{L}(z_{k+1}, \lambda_{k+1})\|^2 - \frac{3\sqrt{\mu}}{4}\left(\frac{1}{\sqrt{s}} - \beta\sqrt{\mu} - \beta^2 \mu \sqrt{s}\right)\|x_{k+1}-x_k\|^2 \\
			& - \frac{\mu}{2}\|z_{k+1}-z_k\|^2 - (1- \beta\sqrt{\mu s}) \langle A^\top (\lambda_{k+1} - \lambda^*), x_{k+1}-x_k \rangle \\
			& + \beta s \langle A^\top (\lambda_{k+1} - \lambda^*), \nabla_x \mathcal{L}(z_{k+1}, \lambda_{k+1}) \rangle - \sqrt{\mu s} \langle A^\top (\lambda_{k+1} - \lambda^*), x_{k+1}-x^* \rangle \\
			= &~ -\frac{s}{6}(1+4\beta)\|\nabla_x \mathcal{L}(z_{k+1}, \lambda_{k+1})\|^2 - \frac{3\sqrt{\mu}}{4}\Big(\frac{1}{\sqrt{s}} - \beta\sqrt{\mu} - \beta^2 \mu \sqrt{s}\Big)\|x_{k+1}-x_k\|^2 \\
			& - \frac{\mu}{2}\|z_{k+1}-z_k\|^2 - \sqrt{\mu s} \left\langle \lambda_{k+1} - \lambda^*, A z_{k+1} -b + \frac{1}{\sqrt{\mu s}} A(z_{k+1}-z_k) \right\rangle \\
			= &~ -\frac{s}{6}(1+4\beta)\|\nabla_x \mathcal{L}(z_{k+1}, \lambda_{k+1})\|^2 - \frac{3\sqrt{\mu}}{4}\Big(\frac{1}{\sqrt{s}} - \beta\sqrt{\mu} - \beta^2 \mu \sqrt{s}\Big)\|x_{k+1}-x_k\|^2 \\
			& - \frac{\mu}{2}\|z_{k+1}-z_k\|^2 -\frac{\sqrt{\mu s}}{\eta_k} \langle \lambda_{k+1} - \lambda^*, \lambda_{k+1} - \lambda_k \rangle,
		\end{split}
	\end{equation*}
	where the first equality holds because $-s\nabla_x\mathcal{L}(z_{k+1}, \lambda_{k+1})=(1+2\sqrt{\mu s})(x_{k+1}-x_k)-(x_k-x_{k-1})$ (see \eqref{@@}), and the last equality follows from Step 3 of Algorithm \ref{kuangjia}. The proof is complete. \qed
	
\end{proof}

\begin{theorem} \label{lisan}
	Let $\{(x_k,\lambda_k)\}_{k \geq 1}$ be the sequence generated by Algorithm \textup{\ref{kuangjia}}. Suppose that $0 \leq \beta \leq \frac{\sqrt{5}-1}{2\sqrt{\mu s}}$, $\eta_k \leq \frac{2+\sqrt{\mu s}}{2} \eta_{k-1}$ and $\lim \limits_{k \to +\infty} \eta_k = + \infty$. Then, for any $(x^*, \lambda^*)\in \Omega$, the sequence $\{(x_k, \lambda_k)\}_{k \geq 1}$ is bounded, and the following results are satisfied:
		\begin{align*}
			\left\{
			\begin{array}{ll}
				|f(x_k) - f(x^{*})| = \mathcal{O} \left(\frac{1}{\sqrt{\eta_{k-1}}}\right),~\text{as}~k \rightarrow +\infty,\\ [3mm]
				\|Ax_k-b\| = \mathcal{O} \left(\frac{1}{\sqrt{\eta_{k-1}}}\right),~\text{as}~k \rightarrow +\infty,\\ [3mm]
				\|x_k - x^{*}\| = \mathcal{O} \left(\frac{1}{\sqrt{\eta_{k-1}}}\right),~\text{as}~k \rightarrow +\infty,\\ [3mm]
				\|x_k - x_{k-1}\| = \mathcal{O} \left(\frac{1}{\sqrt{\eta_{k-1}}}\right),~\text{as}~k \rightarrow +\infty.
			\end{array}
			\right.
		\end{align*}
\end{theorem}

\begin{proof}
	We first introduce the energy sequence, denoted as $E_k$, as follows:
	\begin{align*}
		E_k := \eta_{k-1} W_k + \frac{\sqrt{\mu s}}{2+\sqrt{\mu s}}\|\lambda_k-\lambda^{*}\|^2,
	\end{align*}
	where $W_k$ is defined as in \eqref{hfunc2}. From \eqref{hinequ2}, it is clear that
	\begin{align} \label{lisane}
		\begin{split}
			E_{k+1} - E_k = &~ \eta_k W_{k+1} - \eta_{k-1} W_k + \frac{\sqrt{\mu s}}{2+\sqrt{\mu s}}(\|\lambda_{k+1}-\lambda^{*}\|^2 - \|\lambda_k-\lambda^{*}\|^2) \\
			\leq &~ \left( \frac{2\eta_k}{2+\sqrt{\mu s}} - \eta_{k-1} \right) W_k - \frac{s(1+4\beta)\eta_k}{3(2+\sqrt{\mu s})}\|\nabla_x \mathcal{L}(z_{k+1}, \lambda_{k+1})\|^2 \\
			& - \frac{3\sqrt{\mu}\eta_k}{2(2+\sqrt{\mu s})}\left(\frac{1}{\sqrt{s}} - \beta\sqrt{\mu} - \beta^2 \mu \sqrt{s}\right)\|x_{k+1} - x_k\|^2 \\
			& - \frac{\mu\eta_k}{2+\sqrt{\mu s}} \|z_{k+1} - z_k\|^2 - \frac{2\sqrt{\mu s}}{2+\sqrt{\mu s}} \langle \lambda_{k+1} - \lambda_k, \lambda_{k+1} - \lambda^* \rangle \\
			& + \frac{\sqrt{\mu s}}{2+\sqrt{\mu s}}(\|\lambda_{k+1}-\lambda^{*}\|^2 - \|\lambda_k-\lambda^{*}\|^2)\\
			= &~ \left( \frac{2\eta_k}{2+\sqrt{\mu s}} - \eta_{k-1} \right) W_k - \frac{s(1+4\beta)\eta_k}{3(2+\sqrt{\mu s})}\|\nabla_x \mathcal{L}(z_{k+1}, \lambda_{k+1})\|^2 \\
			& - \frac{3\sqrt{\mu}\eta_k}{2(2+\sqrt{\mu s})}\left(\frac{1}{\sqrt{s}} - \beta\sqrt{\mu} - \beta^2 \mu \sqrt{s}\right)\|x_{k+1} - x_k\|^2 \\
			& - \frac{\mu\eta_k}{2+\sqrt{\mu s}} \|z_{k+1} - z_k\|^2 - \frac{\sqrt{\mu s}}{2+\sqrt{\mu s}}\|\lambda_{k+1}-\lambda_k\|^2.
		\end{split}
	\end{align}
	Since $0 \leq \beta \leq \frac{\sqrt{5}-1}{2\sqrt{\mu s}}$, we have $\frac{1}{\sqrt{s}} - \beta\sqrt{\mu} - \beta^2 \mu \sqrt{s} \geq 0$. This together with $\eta_k \leq \frac{2+\sqrt{\mu s}}{2} \eta_{k-1}$ and \eqref{lisane} yields $E_{k+1} - E_k \leq 0$. It means that $$E_k \leq E_1.$$ Clearly, from $\eta_{k-1}W_k \leq E_k \leq E_1$, it follows that the sequence $\{(x_k, \lambda_k)\}_{k\geq 1}$ is bounded and the following results are satisfied:
	\begin{align}
		& \mathcal{L}(z_k, \lambda^*)-\mathcal{L}(x^*, \lambda^*) = \mathcal{O} \left(\frac{1}{\eta_{k-1}}\right), ~\text{as}~ k \rightarrow +\infty, \label{l2} \\
		& \left\|\sqrt{\mu}(x_k-x^*) + \frac{1}{\sqrt{s}}(x_k - x_{k-1})\right\| =  \mathcal{O} \left(\frac{1}{\sqrt{\eta_{k-1}}}\right), ~\text{as}~ k \rightarrow +\infty, \label{xx2} \\
		& \|x_k -x^*\| = \mathcal{O} \left(\frac{1}{\sqrt{\eta_{k-1}}}\right), ~\text{as}~ k \rightarrow +\infty, \label{x2}
	\end{align}
	and
	\begin{align} \label{lam2}
		\|\lambda_k - \lambda^*\|^2 \leq \frac{2+\sqrt{\mu s}}{\sqrt{\mu s}}E_1.
	\end{align}
	Clearly, combining $\|x_k - x_{k-1}\| \leq \|\sqrt{\mu s}(x_k-x^*) + x_k - x_{k-1}\| + \sqrt{\mu s}\|x_k-x^*\| $ with \eqref{xx2} and \eqref{x2}, we deduce that
	\begin{align}\label{xd2}
		\begin{split}
			\|x_k - x_{k-1}\| = \mathcal{O} \left(\frac{1}{\sqrt{\eta_{k-1}}}\right),~\text{as}~ k \rightarrow +\infty.
		\end{split}
	\end{align}
	From \eqref{lam2}, we obtain $$\|\lambda_{k+1}-\lambda_1\|\leq\|\lambda_{k+1}-\lambda^*\|+\|\lambda_1-\lambda^*\|\leq 2\sqrt{\frac{2+\sqrt{\mu s}}{\sqrt{\mu s}}E_1}.$$ Furthermore, it follows from Step 3 of Algorithm \ref{kuangjia} that
	\begin{align}
		\begin{split} \label{lamda2}
			\sqrt{\mu s} (\lambda_{k+1} - \lambda_1) = &~ \sqrt{\mu s} \sum_{i=1}^k(\lambda_{i+1} - \lambda_i) \\
			= &~ \sqrt{\mu s} \sum_{i=1}^k \eta_i(Az_{i+1}-b) + \sum_{i=1}^k \eta_i A(z_{i+1} - z_i) \\
			= &~ \sum_{i=1}^k \eta_i \big[ (1+\sqrt{\mu s})(Az_{i+1}-b) - (Az_i-b) \big] \\
			= &~ (1+\sqrt{\mu s}) \big[\eta_k(Az_{k+1}-b) - \eta_0(Az_1-b) \big] \\
			&~ + \sum_{i=1}^k \big( \eta_{i-1}(1+\sqrt{\mu s}) - \eta_i \big)(Az_i-b).
		\end{split}
	\end{align}
	Let $r_k := \eta_{k-1}(Az_k-b)$, $b_k = \dfrac{\eta_{k-1} - \frac{\eta_k}{1+\sqrt{\mu s}}}{\eta_{k-1}} \in [0,1)$ and $C = 2\frac{\sqrt{\mu s}}{1+ \sqrt{\mu s}}\sqrt{\frac{2+\sqrt{\mu s}}{\sqrt{\mu s}}E_1} + \eta_0\|Az_1-b\|$. Together with \eqref{lamda2} and Lemma \ref{yinli2}, we have $\eta_{k-1} \|Az_k-b\| < \eta_0 \|Az_1 -b\| + 2C$, which means that
	\begin{align}\label{ad2}
		\|Az_k-b\| = \mathcal{O} \left(\frac{1}{\eta_{k-1}}\right),~\text{as}~ k \rightarrow +\infty.
	\end{align}
	This together with $\|Ax_k-b\| \leq \|Az_k - b\| + \beta \| A \|\|x_k - x_{k-1}\|$ and \eqref{xd2} yields
	\begin{align*}
		\|Ax_k - b\| = \mathcal{O} \left(\frac{1}{\sqrt{\eta_{k-1}}}\right),~\text{as}~ k \rightarrow +\infty.
	\end{align*}
	Note that
	\begin{equation*}
		\begin{split}
			|f(z_k) - f(x^{*})|
			\leq &~ \mathcal{L}(z_k, \lambda^*)-\mathcal{L}(x^*, \lambda^*) +  |\langle \lambda^{*}, Az_k-b \rangle|\\
			\leq &~ \mathcal{L}(z_k, \lambda^*)-\mathcal{L}(x^*, \lambda^*) + \|\lambda^{*}\| \|Az_k-b\|.
		\end{split}
	\end{equation*}
	Thus, it follows from \eqref{l2} and \eqref{ad2} that
	\begin{align}\label{ff2}
		|f(z_k) - f(x^{*})| = \mathcal{O} \left(\frac{1}{\eta_{k-1}}\right),~\text{as}~ k \rightarrow +\infty.
	\end{align}
	Taking \eqref{L} into account, it is easy to show that
	\begin{align*}
		f(x_k) - f(z_k) \leq &~ \langle \nabla f(z_k), -\beta(x_k - x_{k-1}) \rangle + \frac{L\beta^2}{2}\|x_k - x_{k-1}\|^2 \\
		\leq &~ \beta \|\nabla f(z_k)\| \|x_k - x_{k-1}\| + \frac{L\beta^2}{2}\|x_k - x_{k-1}\|^2
	\end{align*}
	and
	\begin{align*}
		f(z_k) - f(x_k) \leq &~ \langle \nabla f(x_k), \beta(x_k - x_{k-1}) \rangle + \frac{L\beta^2}{2}\|x_k - x_{k-1}\|^2 \\
		\leq &~ \beta \|\nabla f(x_k)\| \|x_k - x_{k-1}\| + \frac{L\beta^2}{2}\|x_k - x_{k-1}\|^2.
	\end{align*}
	Combining the Lipschitz continuity of $\nabla f$, \eqref{xd2} and the boundedness of $\{x_k\}_{k\geq 1}$, we obtain
	\begin{align*}
		|f(x_k) - f(z_k)|  = \mathcal{O} \left(\frac{1}{\sqrt{\eta_{k-1}}}\right),~\text{as}~ k \rightarrow +\infty.
	\end{align*}
	 This together with \eqref{ff2} and $|f(x_k)-f(x^*)| \leq |f(x_k) - f(z_k)| + |f(z_k) - f(x^*)|$ implies
	\begin{align*}
	 |f(x_k)-f(x^*)| = \mathcal{O} \left(\frac{1}{\sqrt{\eta_{k-1}}}\right),~\text{as}~ k \rightarrow +\infty.
	\end{align*}
	 The proof is complete. \qed
\end{proof}

\begin{remark}
	\begin{enumerate}
		\item[{\rm (i)}] \textup{From Theorems \ref{lianxu} and \ref{lisan}, it is clear that the convergence rate of Algorithm \ref{kuangjia} matches that of the continuous-time dynamical system \eqref{system}.}
		\item[{\rm (ii)}] \textup{In Theorem \ref{lisan}, the condition $\eta_k\leq  \frac{2+\sqrt{\mu s}}{2} \eta_{k-1}$ can be rewritten as
		\begin{align*}
			\frac{\eta_k-\eta_{k-1}}{\sqrt{s}}\leq \frac{\sqrt{\mu}}{2+\sqrt{\mu s}}\eta_k\leq \frac{\sqrt{\mu}}{2}\eta_k.
		\end{align*}
		This can be viewed as a discretized version of $\dot{ \eta}(t)\leq \frac{\sqrt{\mu}}{2}\eta(t)$,  which is consistent with the condition in Theorem \ref{lianxu}.}
	\end{enumerate}
\end{remark}

In the specific case where $\eta_k = \frac{2+\sqrt{\mu s}}{2}\eta_{k-1}$, we obtain $\eta_k = \left(\frac{2+\sqrt{\mu s}}{2}\right)^{k}\eta_0$. Then, by applying Theorem \ref{lisan}, we establish the following results.

\begin{corollary}\label{coro2}
	Let $\{(x_k,\lambda_k)\}_{k \geq 1}$ be the sequence generated by Algorithm \textup{\ref{kuangjia}}. Suppose that $0 \leq \beta \leq \frac{\sqrt{5}-1}{2\sqrt{\mu s}}$ and $\eta_k = \left(\frac{2+\sqrt{\mu s}}{2}\right)^{k} \eta_0$. Then, for any $(x^*, \lambda^*)\in \Omega$, the sequence $\{(x_k,\lambda_k)\}_{k \geq 1}$ is bounded, and the following results are satisfied:
	\begin{align*}
		\left\{
		\begin{array}{lll}
			|f(x_k) - f(x^{*})| = \mathcal{O} \Big(\frac{2} {2+\sqrt{\mu s} }\Big)^{\frac{k-1}{2}}, ~\text{as}~k \rightarrow +\infty, \\ [3mm]
			\|Ax_k-b\| = \mathcal{O} \Big(\frac{2} {2+\sqrt{\mu s} }\Big)^{\frac{k-1}{2}},~\text{as}~k \rightarrow +\infty,\\ [3mm]
			\|x_k- x^{*}\| = \mathcal{O} \Big(\frac{2} {2+\sqrt{\mu s} }\Big)^{\frac{k-1}{2}},~\text{as}~k \rightarrow +\infty,\\ [3mm]
			\|x_k- x_{k-1}\| = \mathcal{O} \Big(\frac{2} {2+\sqrt{\mu s} }\Big)^{\frac{k-1}{2}},~\text{as}~k \rightarrow +\infty.
		\end{array}
		\right.
	\end{align*}
\end{corollary}

\section{Numerical Experiments}

	In this section, we design numerical experiments to verify the theoretical results underlying Algorithm \ref{kuangjia}. All numerical experiments are conducted using MATLAB R2020b on a Windows 10 PC equipped with an Intel Core i5-1135G7 processor (2.40 GHz, 1.38 GHz).
	
	\begin{example}\cite[Example 4.2]{he2026} Consider the Distributed Logistic regression with $\ell_2$ regularization:
	\begin{align} \label{exp1}
		\left\{ \begin{array}{cc}
			\min_{x \in \mathbb{R}^{pm}} & f(x) = \sum_{i=1}^{p} \Big[ \ln\big(1 + e^{-c_i^\top x_i}\big) + \frac{\epsilon_i}{2} \|x_i\|^2 \Big] \\
			s.t.& x_i = x_j, ~~\forall (i, j) \in \mathcal{E}.
		\end{array}
		\right.
	\end{align}
	Here, $\mathcal{G}(\mathcal{V}, \mathcal{E}, H)$ denotes an undirected network with $p$ nodes, in which $\mathcal{V}$, $\mathcal{E}$ and $H$ stand for the vertex set, edge set and doubly stochastic weight matrix, respectively; $x = [x_1^\top, x_2^\top, \ldots, x_p^\top]^\top$, $x_i \in \mathbb{R}^m$, $c_i \in \mathbb{R}^m$ and $\epsilon_i > 0$, $i=1,\ldots,p$. Note that Problem \eqref{exp1} has broad applications in many fields, such as machine learning and statistical learning \cite{exp, ring}. For the ring-topology undirected connected network with $p$ agents, the consensus constraints $x_1 = \cdots = x_p$ are equivalent to the linear equality constraint $L_{p\circ m} x = \mathbf{0}_{pm}$, with $L_{p\circ m} = (I_p - H) \otimes I_m$. Consequently, the original problem \eqref{exp1} can be reformulated as:
	\begin{align}
		\left\{ \begin{array}{cc}
			\min_{x \in \mathbb{R}^{pm}}& f(x) = \sum_{i=1}^{p} f_i (x_i)  \\
			s.t.& L_{p\circ m} x = \mathbf{0}_{pm},
		\end{array}
		\right.
	\end{align}
	where $f_i (x_i) = \log\Big(1 + e^{-c_i^\top x_i}\Big) + \frac{\epsilon_i}{2} \|x_i\|^2$. Let each component of $c_i$ be independently and identically drawn from a uniform distribution on $[0,1]$, and let the scalars $\epsilon_i$ be independently sampled from a uniform distribution on $[4,6]$. Under these settings, $f$ is $\mu$-strongly convex with an $L$-Lipschitz continuous gradient, where $\mu = \min \{ \epsilon_i|i=1,\ldots,p\}$ and $L = \max \Big\{ \frac{\|c_i\|^2}{4} + \epsilon_i | i=1,\ldots,p \Big\}$.
	
    It should be noted that, in practice, it is difficult to exactly solve the subproblem (Step 2 of Algorithm \ref{kuangjia}). Although the theoretical convergence of Algorithm \ref{kuangjia} relies on the assumption of exact subproblem solving, a sufficiently accurate approximate solution can be treated as a small perturbation of the exact one. As similarly demonstrated in \cite{ZH, he2026, hex2025}, if the error is properly controlled, such a mild approximation error does not affect the convergence performance. In this case, the subproblem occurring in Algorithm \ref{kuangjia} is solved approximately with a sufficiently small stopping tolerance, and the stopping condition is
    \begin{equation*}
        \frac{\|x_{k} - x_{k-1}\|}{\max\{\|x_{k-1}\|,1\})} \le 10^{-6}
    \end{equation*}
    or the maximum inner iterations reach $150$.

    In the first numerical experiment, we choose three different values of the parameter of $\eta_k$ for Algorithm \ref{kuangjia} (denoted by AL1). Then, we compare the performance of Algorithm \ref{kuangjia}, the accelerated linearized augmented Lagrangian method (ALALM-Xu) (\cite{xu}, Algorithm 1) and the accelerated linearized primal-dual method (ALPDM) (\cite{he2022}, Algorithm 2). The maximum number of iterations is set to 800, and the parameter settings for these algorithms are as follows:
	
	\begin{itemize}
		\item ALALM-Xu: $\alpha_k=\frac{2}{k+2}$, $\eta_0 = 1$, $\eta_k = \gamma_k = 50k$, and $P_k = \frac{2L}{k}I_{mp}$.
		\item ALPDM: $s=2m$, $\alpha= 10$, and $M = sL I_{mp}$.
		\item AL1: $s=\frac{1}{L}$, $\eta_0 = 1$, $\beta = \frac{1}{3\sqrt{\mu s}}$. For $\eta_k$, we consider the following three cases:
	\begin{align*}
		\text{Case 1}:&~ \eta_k = \min \left\{k^2,~ \frac{2+\sqrt{\mu s}}{2} \eta_{k-1} \right\};\\
		\text{Case 2}:&~ \eta_k = \min \left\{k^3,~ \frac{2+\sqrt{\mu s}}{2} \eta_{k-1} \right\};\\
		\text{Case 3}:&~ \eta_k = \Big(\frac{2+\sqrt{\mu s}}{2}\Big)^{k} \eta_0.
	\end{align*}
	\end{itemize}
	The experimental results are shown in Figures \ref{fig1} and \ref{fig2}.
	\begin{figure}[H]
		\centering
		\begin{minipage}{0.45\textwidth}
			\centering
			\includegraphics[width=\linewidth]{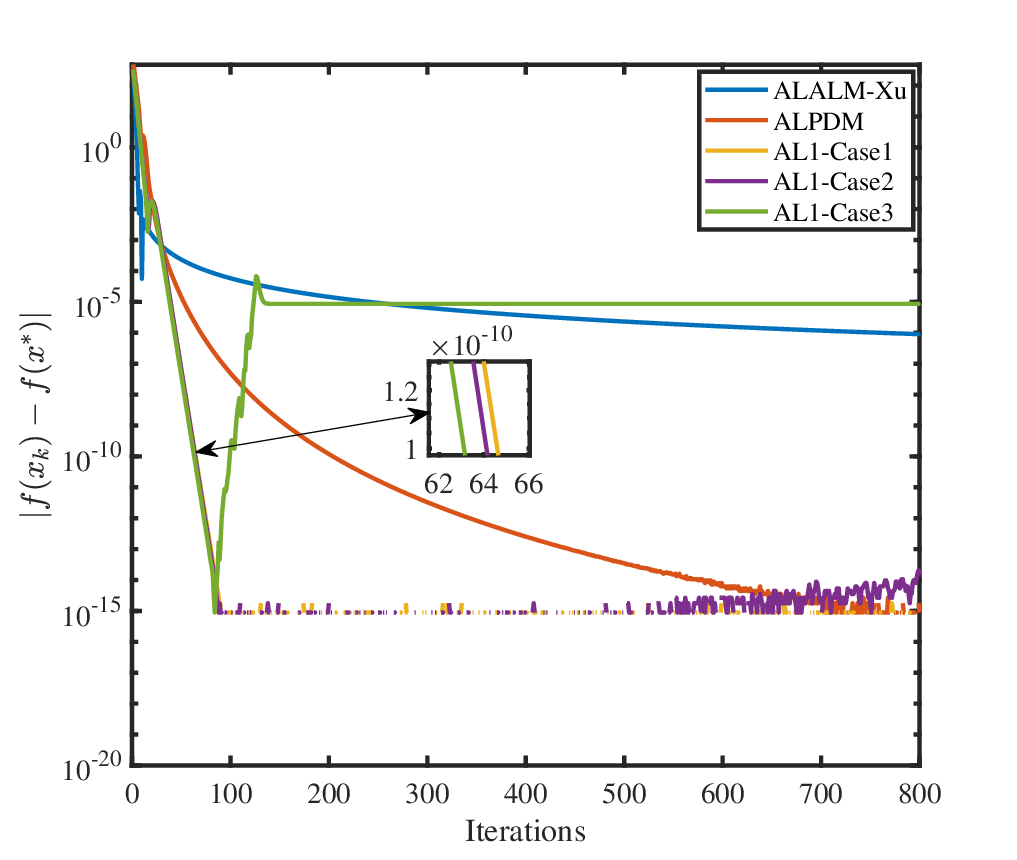}
		\end{minipage}
		\hfill
		\begin{minipage}{0.45\textwidth}
			\centering
			\includegraphics[width=\linewidth]{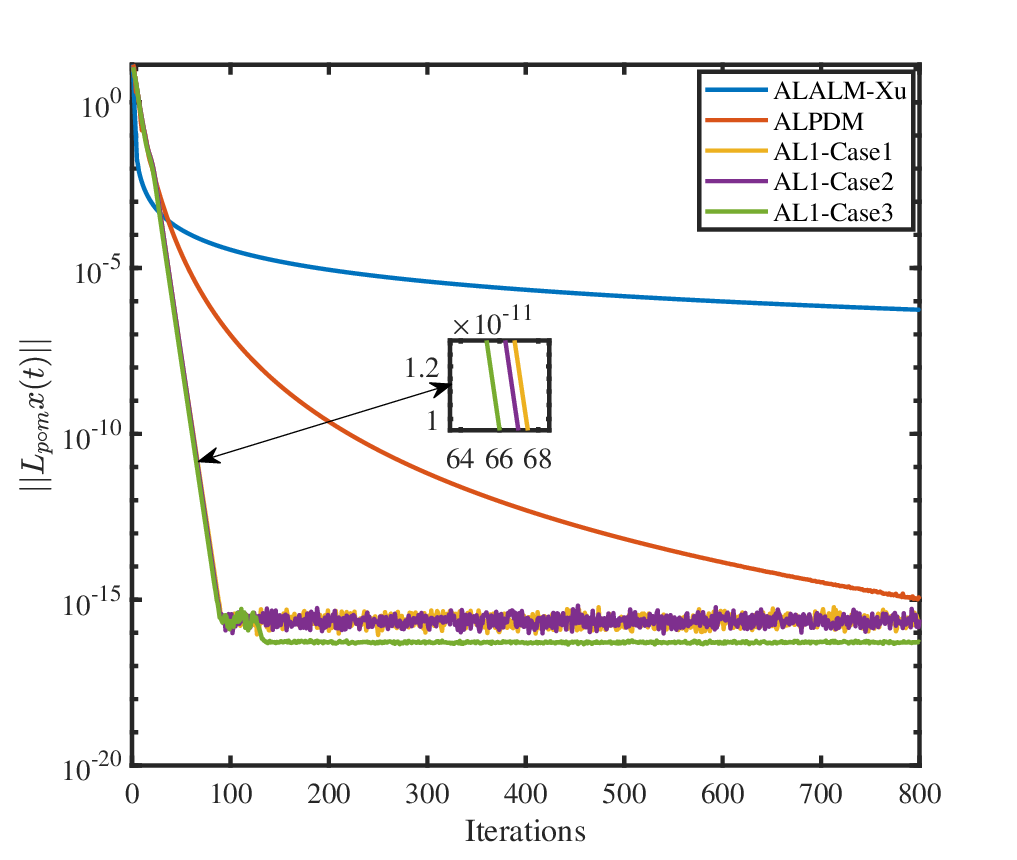}
		\end{minipage}
		\caption{Objective residual and feasibility violation of different algorithms for the case with parameters $(p,m)=(10, 30)$.}
		\label{fig1}
	\end{figure}
	
	\begin{figure}[H]
		\centering
		\begin{minipage}{0.45\textwidth}
			\centering
			\includegraphics[width=\linewidth]{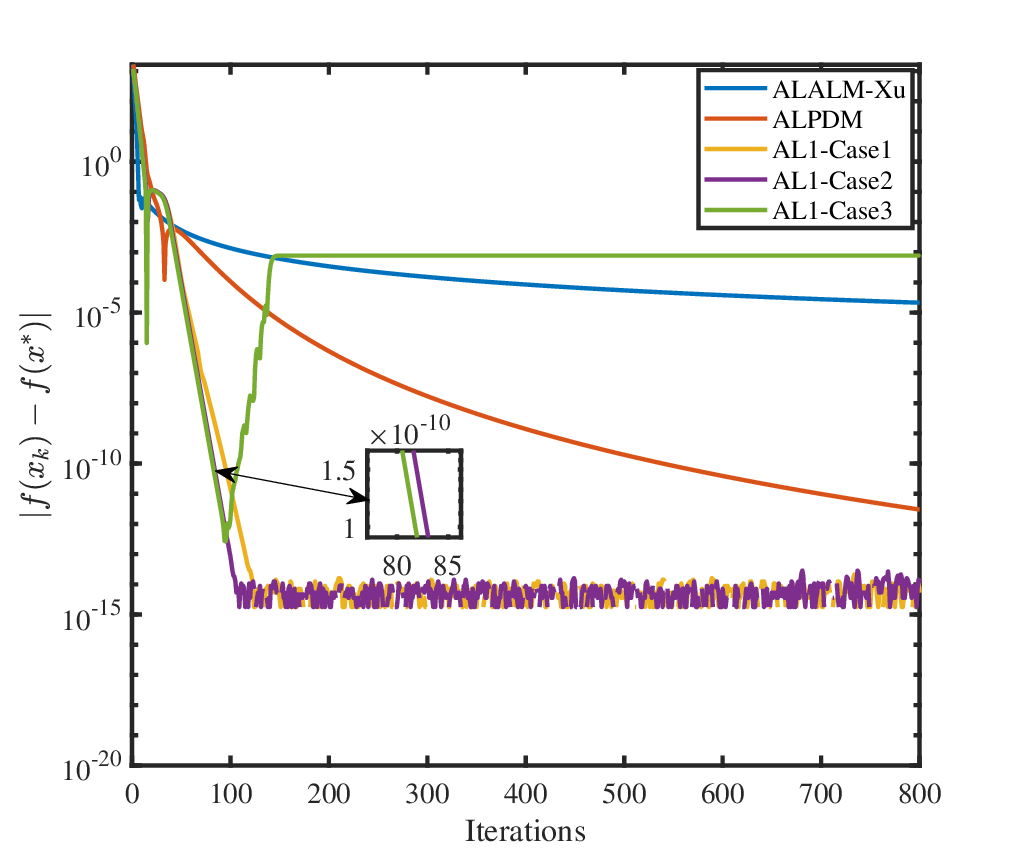}
		\end{minipage}
		\hfill
		\begin{minipage}{0.45\textwidth}
			\centering
			\includegraphics[width=\linewidth]{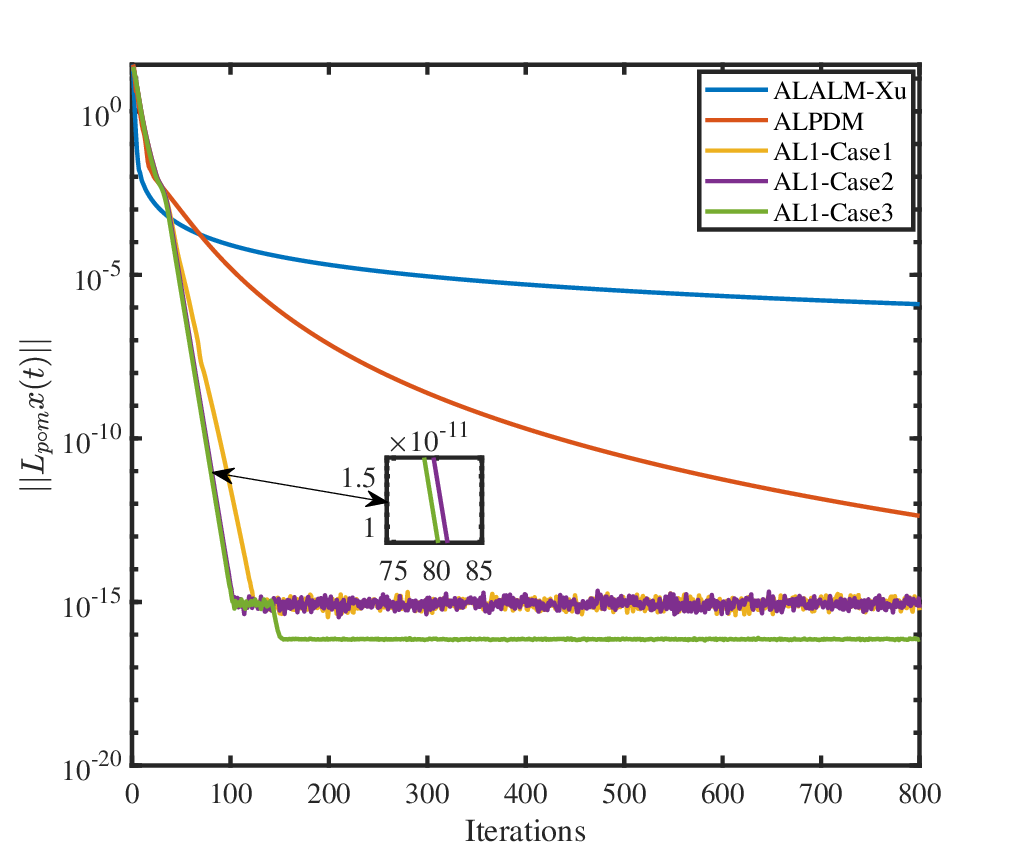}
		\end{minipage}
		\caption{Objective residual and feasibility violation of different algorithms for the case with parameters $(p,m)=(20, 50)$.}
		\label{fig2}
	\end{figure}
	As shown in Figures \ref{fig1} and \ref{fig2}, our algorithm, AL1, exhibits superior convergence performance. Furthermore, comparing the algorithms under different values of $\eta_k$, it is found that AL1 performs best overall under Case 2. Theoretically, the fastest convergence rate should be achieved when $\eta_k$ follows Case 3. However, a fast growing $\eta_k$ makes subproblems difficult to solve. Therefore, the resulting inexact subproblem solutions degrade the algorithm's actual performance. This observation motivates us to select a suitable growth rate for $\eta_k$ in practical applications.
	
	In the following numerical experiments, we test different values of the parameter of $\beta$ for AL1 and compare its performance with ALALM-Xu and ALPDM. For AL1, we choose $\eta_k = \min \Big\{k^3, \frac{2+\sqrt{\mu s}}{2}\eta_{k-1} \Big\}$, and set $\beta = \kappa \beta_{\max}$, where $\kappa \in \{0.3, 0.5, 1\}$ and $\beta_{\max}= \frac{1}{2\sqrt{\mu s}}$. The results are shown in Figures \ref{fig3} and \ref{fig4}.
	\begin{figure}[H]
		\centering
		\begin{minipage}{0.45\textwidth}
			\centering
			\includegraphics[width=\linewidth]{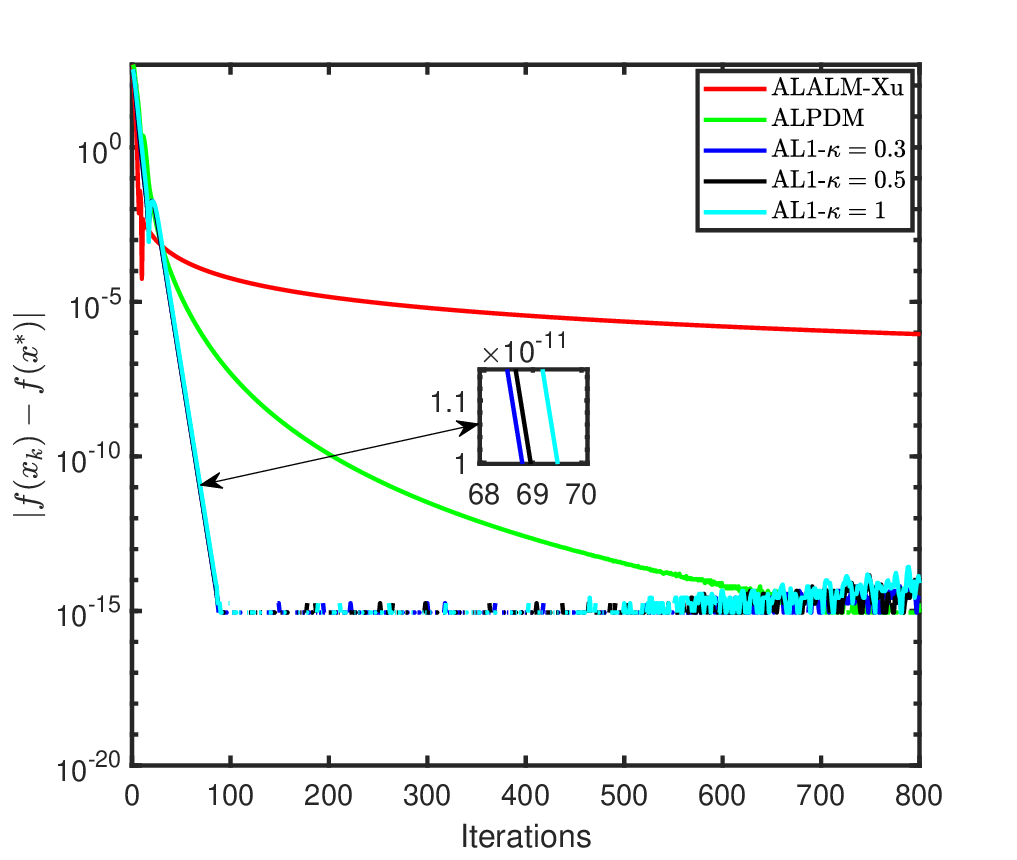}
		\end{minipage}
		\hfill
		\begin{minipage}{0.45\textwidth}
			\centering
			\includegraphics[width=\linewidth]{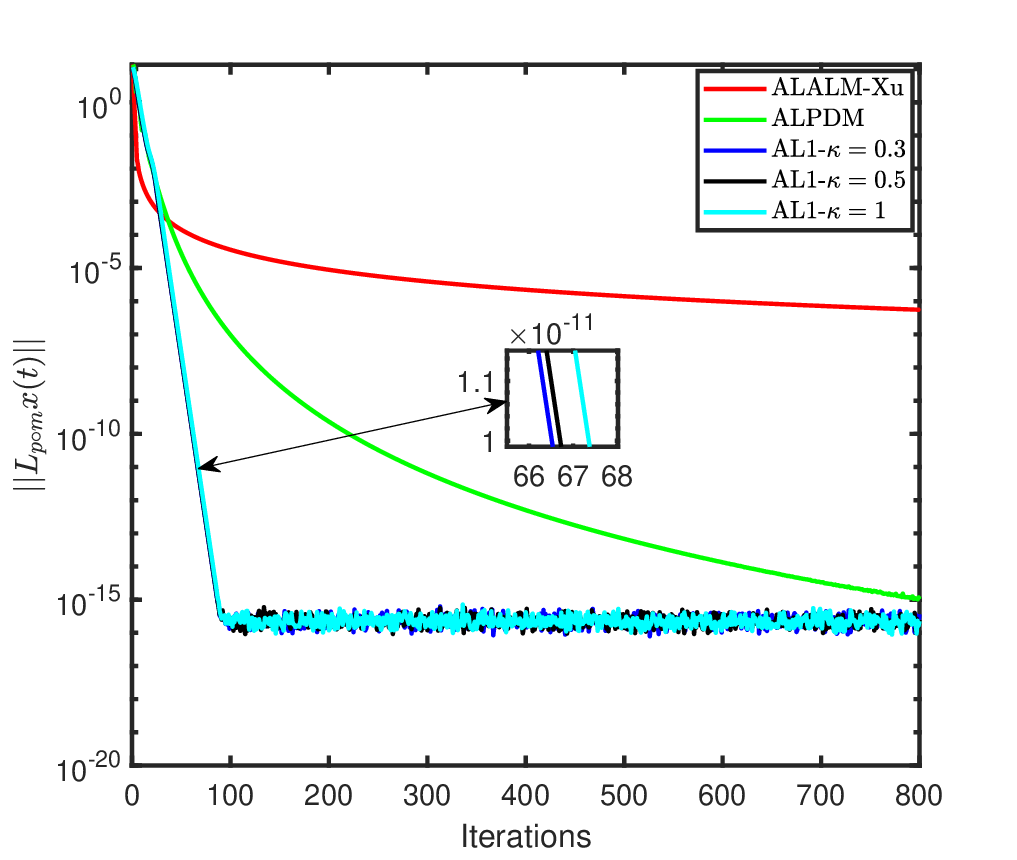}
		\end{minipage}
		\caption{Objective residual and feasibility violation under different $\beta$ for the case with parameters $(p,m)=(10, 30)$.}
		\label{fig3}
	\end{figure}
	
	\begin{figure}[H]
		\centering
		\begin{minipage}{0.45\textwidth}
			\centering
			\includegraphics[width=\linewidth]{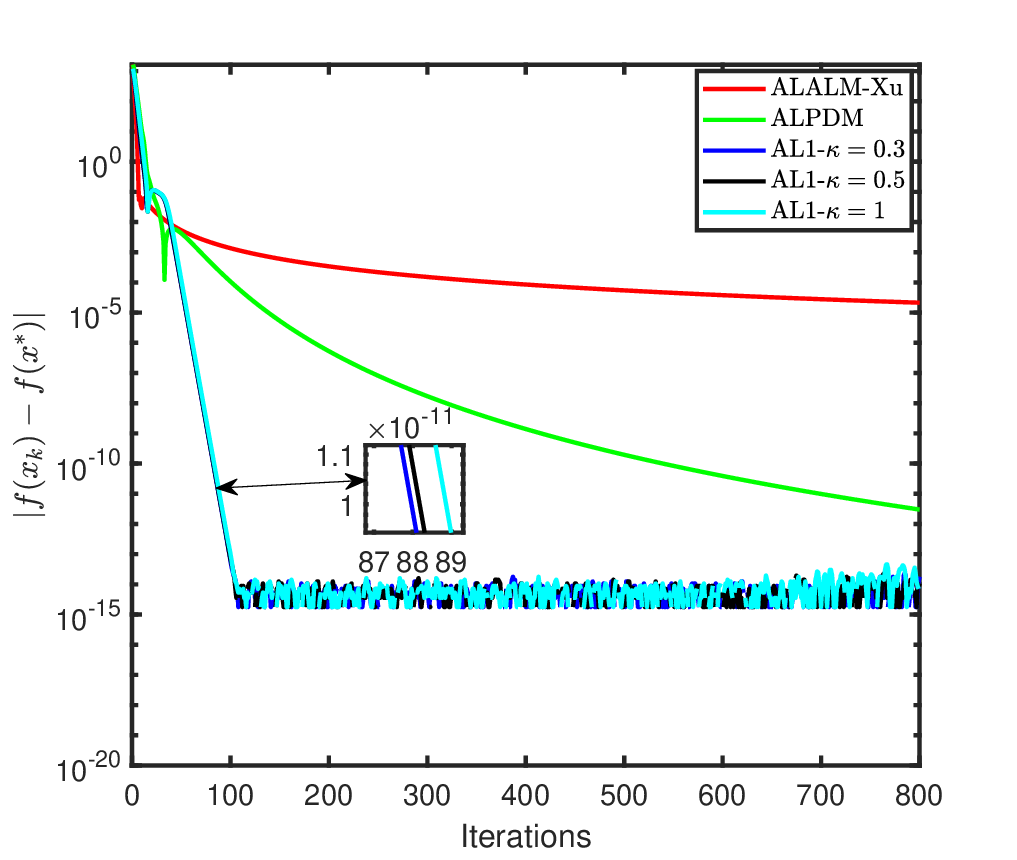}
		\end{minipage}
		\hfill
		\begin{minipage}{0.45\textwidth}
			\centering
			\includegraphics[width=\linewidth]{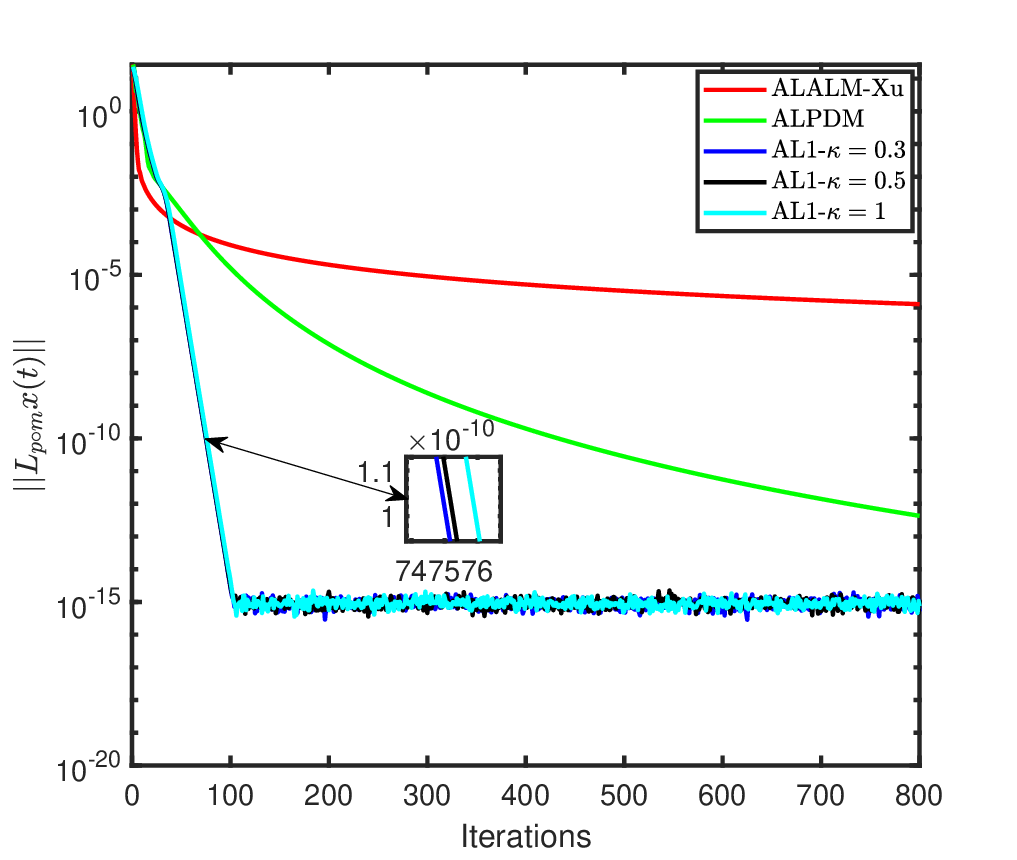}
		\end{minipage}
		\caption{Objective residual and feasibility violation under different $\beta$ for the case with parameters $(p,m)=(20, 50)$.}
		\label{fig4}
	\end{figure}
	As can be observed from Figures \ref{fig3} and \ref{fig4}, the convergence performance of AL1 varies only slightly across different values of $\beta$. In general, the smaller the value of $\beta$, the faster the convergence rate, while the iteration accuracy remains approximately the same.
	
	\end{example}
	
	\begin{example}
		Consider the following regularized least squares problem:
		\begin{equation}
			\min_{x\in\mathbb{R}^n} f(x)=\frac12\|Ax-b\|^2+\frac{\mu}{2}\|x\|^2,
		\end{equation}
		where $A \in \mathbb{R}^{m \times n}$ and $b \in \mathbb{R}^m$. The matrix $A\in \mathbb{R}^{m \times n}$ is randomly generated with sparsity density $\sigma \in (0, 1]$. Its nonzero entries are independently sampled from the uniform distribution over $[-0.1,0.1]$. The vector $b\in \mathbb{R}^m$ is drawn from the standard Gaussian distribution. Under these settings, $f$ is $\mu$-strongly convex with an $L$-Lipschitz continuous gradient, where $\mu>0$ denotes the prescribed Tikhonov regularization parameter and $L = \|A\|^2 + \mu$.
		
		In the following experiments, we compare AL1 with IAPDA proposed in \cite{ZH}, FISTA proposed in \cite{FIS}, the accelerated forward-backward method (AFBM) proposed in \cite{AF} and IPAHD-SC proposed in \cite{firstorder}. Set $\mu=0.01$. The parameter settings for these algorithms are listed as follows:
		\begin{itemize}
			\item IAPDA: $\beta_0 = 1$.
			\item FISTA: $s=\frac{1}{L}$.
			\item AFBM: $s=\frac{1}{L}$ and $\alpha =5$.
			\item AL1: $\sqrt{s} = \frac{1}{3\sqrt{\mu}}$ and $\beta =\frac{1}{6\sqrt{\mu s}}$.
			\item IPAHD-SC: $\sqrt{s} = \beta = \frac{1}{3\sqrt{\mu}}$.
		\end{itemize}
		For the sparsity density parameter $\sigma$, we choose $\sigma = 0.1$ and $\sigma = 0.5$. Furthermore, two different dimension configurations are considered: $(m,n)=(500,1000)$ and $(m,n)=(800,1500)$. The maximum number of iterations is set to $1000$ and $1500$ for the two dimension configurations, respectively. The results are shown in Figures \ref{fig5} and \ref{fig6}.
		\begin{figure}[H]
			     \centering
			     \begin{minipage}{0.45\textwidth}
				    \centering
				    \includegraphics[width=\linewidth]{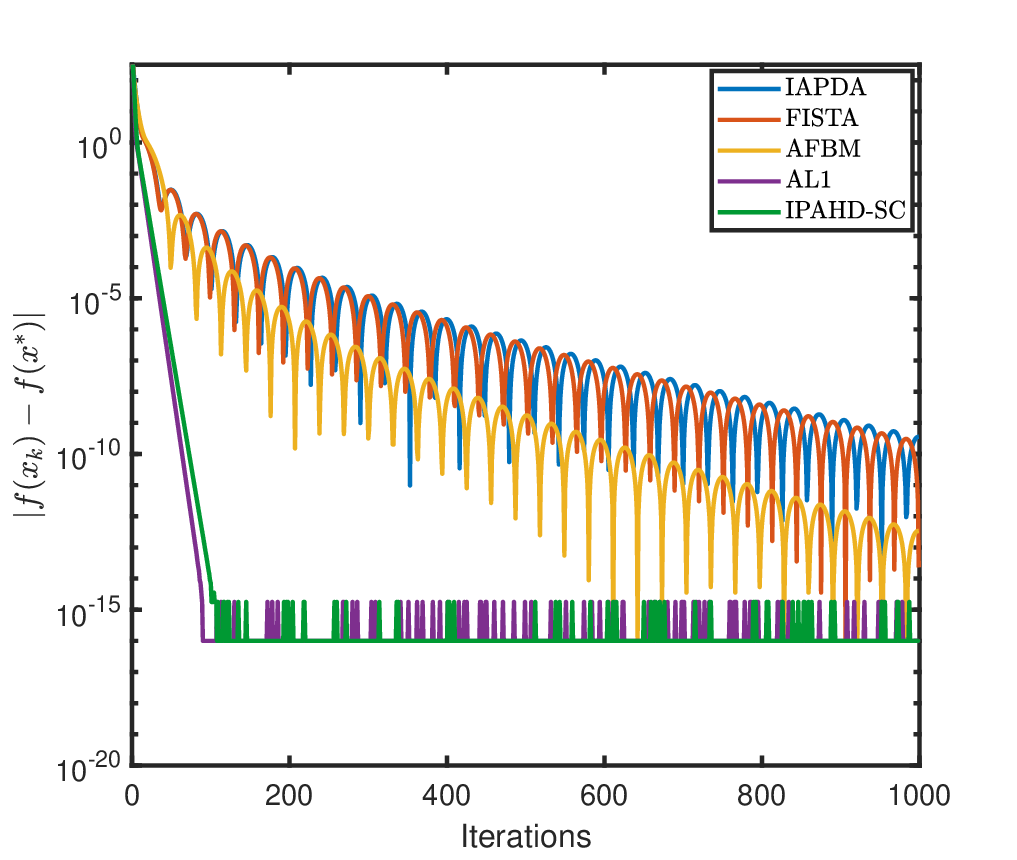}
				    $(a) ~\sigma=0.1$
			     \end{minipage}
			     \hfill
			     \begin{minipage}{0.45\textwidth}
				    \centering
				    \includegraphics[width=\linewidth]{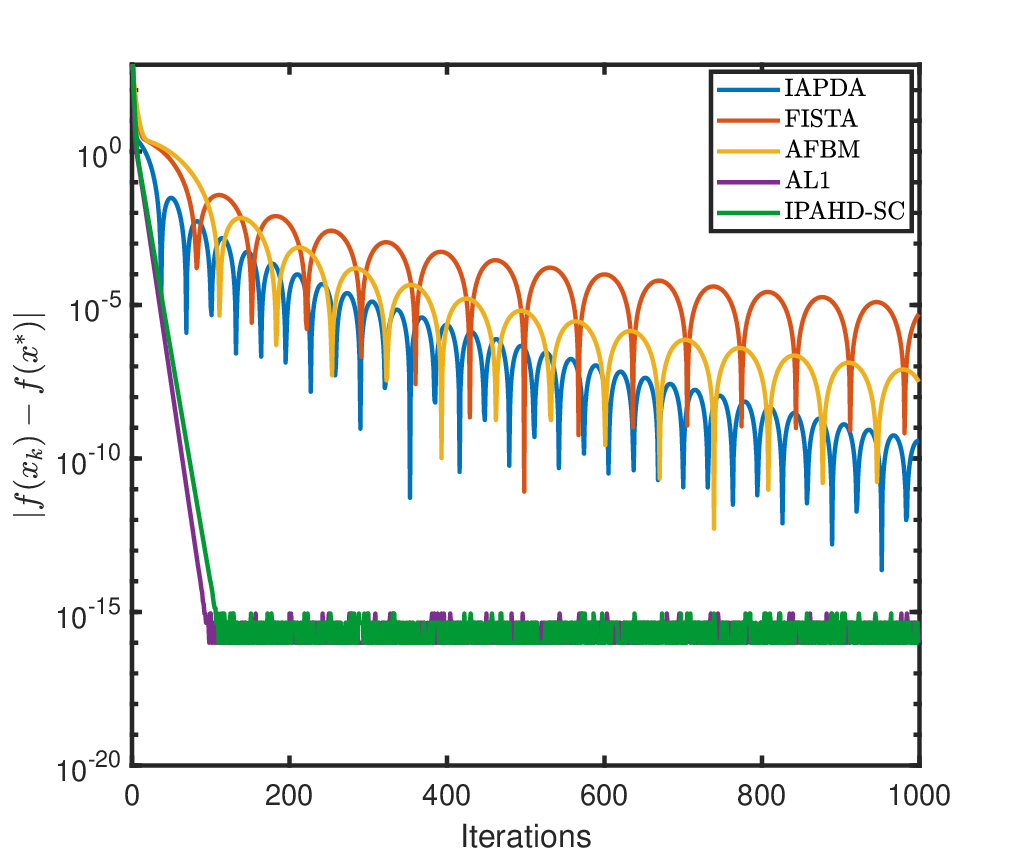}
				    $(b) ~\sigma=0.5$
			     \end{minipage}
			     \caption{Objective residual of different algorithms for the case with parameters $(m,n)=(500, 1000)$.}
			     \label{fig5}
		  \end{figure}
		
		\begin{figure}[H]
			\centering
			\begin{minipage}{0.45\textwidth}
				\centering
				\includegraphics[width=\linewidth]{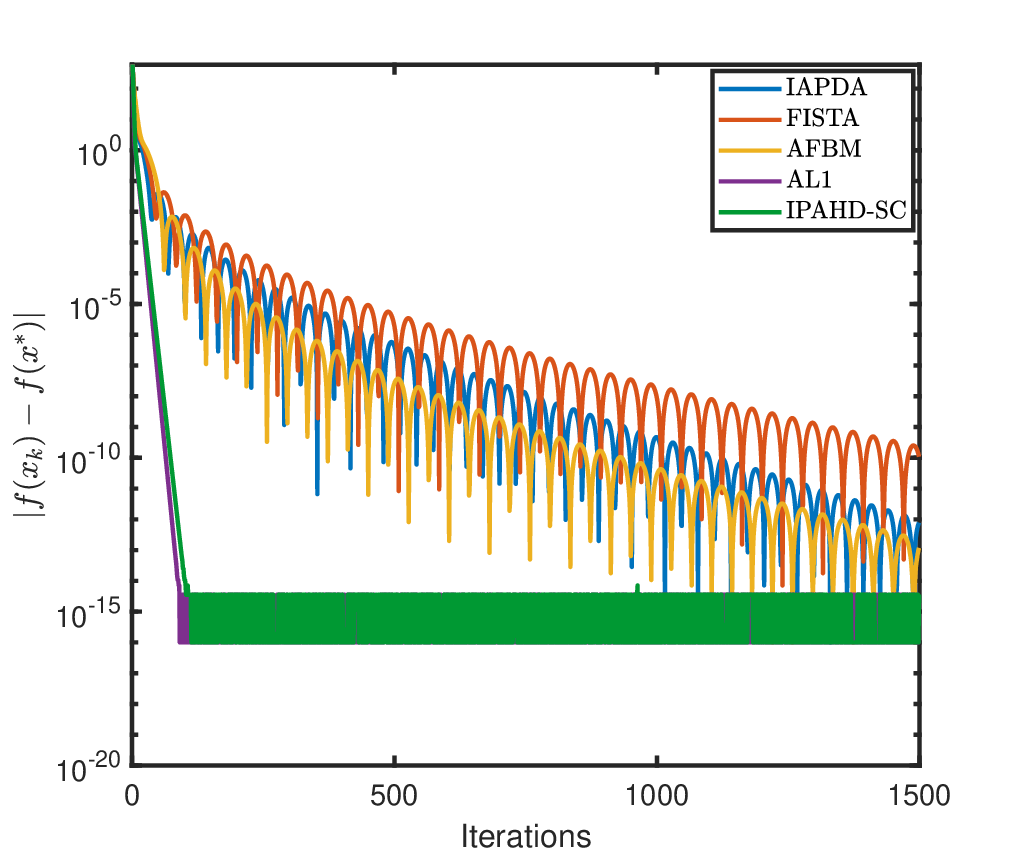}
				$(a) ~\sigma=0.1$
			\end{minipage}
			\hfill
			\begin{minipage}{0.45\textwidth}
				\centering
				\includegraphics[width=\linewidth]{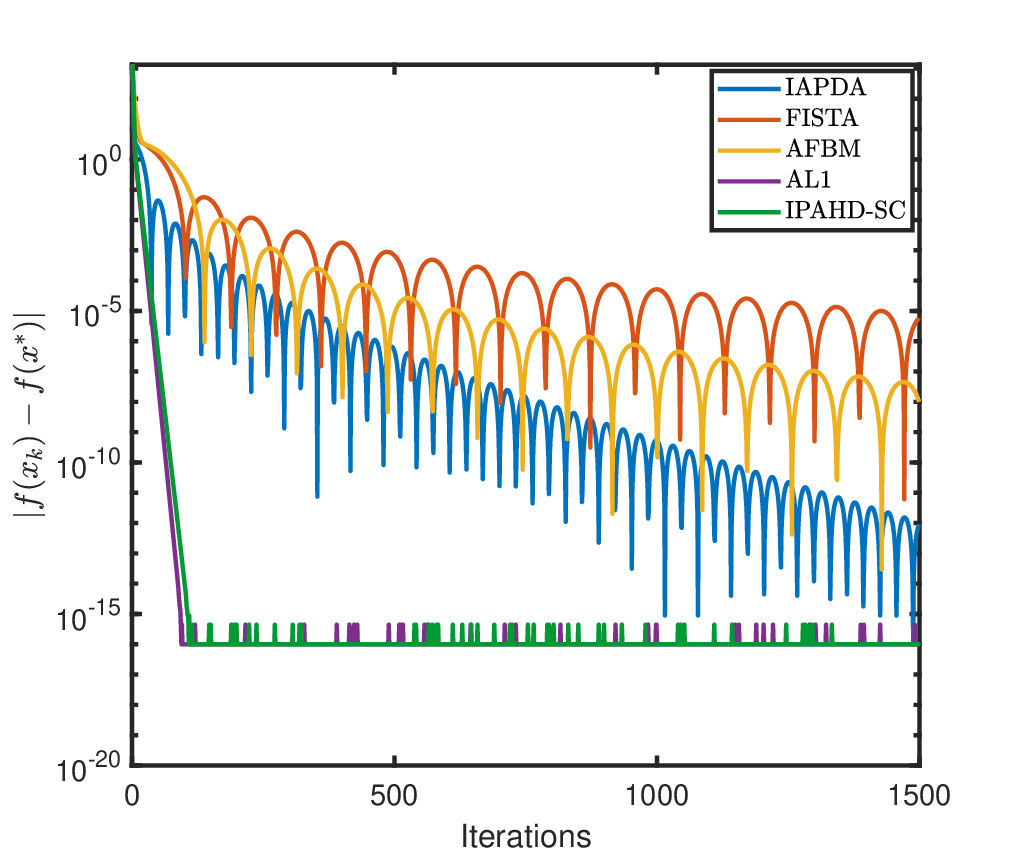}
				$(b) ~\sigma=0.5$
			\end{minipage}
			\caption{Objective residual of different algorithms for the case with parameters $(m,n)=(800, 1500)$.}
			\label{fig6}
		\end{figure}
		From Figures \ref{fig5} and \ref{fig6}, we observe that AL1 and IPAHD-SC achieve higher accuracy, faster convergence rates and less oscillations. Meanwhile, the convergence performances of FISTA and AFBM are affected by $\sigma$.
	\end{example}

\section{Conclusions}
	In this paper, we consider a linear equality constrained convex optimization problem with a strongly convex objective function. We first propose an inertial primal-dual dynamical system \eqref{system} that incorporates an implicit Hessian-driven damping term, and then establish its asymptotic properties. We show that the convergence rates of System \eqref{system} can achieve exponential rates in the best case. By applying a natural implicit time discretization technique, we obtain an inertial accelerated primal-dual algorithm for Problem \eqref{prob}. Using Lyapunov-based techniques, we prove that the proposed algorithm attains fast convergence rates that match those of its continuous-time counterpart.
	
	For future work, it would be of interest to develop an inertial primal-dual dynamical system with asymptotic vanishing damping and implicit Hessian-driven damping for solving ``smooth + nonsmooth" composite convex optimization problems with linear equality constraints. Moreover, whether such a dynamical system can be extended to solve non-convex optimization problems with linear equality constraints is also a highly meaningful research direction.
	
\section*{Appendix}

	\begin{theorem*}[Restatement of Theorem \ref{unique}]
		Suppose that $\nabla f$ is $L$-Lipschitz continuous on $\mathbb{R}^n$ with $L>0$. Let $\eta (t)\in {L}_{loc}^{1}([ t_0,+\infty))$. Then, for any initial conditions $(x(t_0),\dot{x}(t_0), \lambda(t_0)) \in \mathbb{R}^n \times \mathbb{R}^n \times \mathbb{R}^m$, System \eqref{system} admits a unique global strong solution.
	\end{theorem*}
	\begin{proof}
		Let $a(t):=\dot{x}(t)$ and $Z(t):=(x(t), a(t), \lambda(t))$. Then, System \eqref{system} is equivalent to
		$$
		\left\{
		\begin{array}{ll}
			\dot{Z}(t)=G(t,Z(t)),\\
			Z(t_0)=(x(t_0), a(t_0), \lambda(t_0)),
		\end{array}
		\right.
		$$
		where
		\begin{eqnarray*}
			G(t,Z(t))=\begin{pmatrix}
				a(t) \\
				-2\sqrt{\mu}a(t)-\nabla f(x(t)+\beta a(t))-A^\top\lambda(t)\\
				\eta(t)\left[A\Big(x(t) + \big(\frac{1}{\sqrt{\mu}} - \beta\big)a(t)\Big) - b - \frac{\beta A}{\sqrt{\mu}}\big(\nabla f(x(t)+\beta a(t)) + A^\top \lambda(t)\big)\right]
			\end{pmatrix}.
		\end{eqnarray*}
		
		For any $t\geq t_0$, we first show that $G(t,\cdot)$ is $K(t)$-Lipschitz continuous with $K(t)\in L_{\rm loc}^{1}([t_0,+\infty))$. Indeed, for any $Z(t)$ and $\bar{Z}(t)\in\mathbb{R}^n\times\mathbb{R}^n\times\mathbb{R}^m$, we have
		\begin{eqnarray*}
			\begin{split}
				& \|G(t,Z(t))-G(t,\bar{Z}(t))\|\\
				\leq &~ \left(1 + 2\sqrt{\mu} + \Big(\frac{1}{\sqrt{\mu}} - \beta\Big)\eta(t)\|A\| \right) \|a(t)-\bar{a}(t)\| \\
				& + \left(\|A^\top\| + \frac{\beta}{\sqrt{\mu}}\eta(t)\|A A^\top\|\right)\|\lambda(t)-\bar{\lambda}(t)\| +\eta(t)\|A\|\|x(t)-\bar{x}(t)\| \\
				&  +\left(1+\frac{\beta\|A\|}{\sqrt{\mu}}\eta(t)\right)\|\nabla f(x(t)+\beta a(t))-\nabla f(\bar{x}(t)+\beta \bar{a}(t))\|.
			\end{split}
		\end{eqnarray*}
		Since $\nabla f$ is $L$-Lipschitz continuous, the above inequality becomes
		\begin{eqnarray*}
			\begin{split}
				& \|G(t,Z(t))-G(t,\bar{Z}(t))\|\\
				\leq &~ \left(1 + 2\sqrt{\mu} + \Big(\frac{1}{\sqrt{\mu}} - \beta\Big)\eta(t) \|A\| + L\beta\Big(1+\frac{\beta\|A\|}{\sqrt{\mu}}\eta(t)\Big) \right)\|a(t)-\bar{a}(t)\| \\
				& + \left(\|A^\top\| + \frac{\beta}{\sqrt{\mu}}\eta(t)\|AA^\top\|\right)\|\lambda(t)-\bar{\lambda}(t)\|\\
				& +\left( \eta(t)\|A\| + L\Big(1+\frac{\beta\|A\|}{\sqrt{\mu}}\eta(t)\Big) \right)\|x(t)-\bar{x}(t)\| \\
				\leq &~ K(t) \|Z(t)-\bar{Z}(t)\|.
			\end{split}
		\end{eqnarray*}
		where $$ K(t):= \left(1+\frac{1}{\sqrt{\mu}}-\beta + \frac{L\beta(1+\beta)}{\sqrt{\mu}}\right)\eta(t)\|A\| + \|A^\top\| +\frac{\beta}{\sqrt{\mu}}\eta(t)\|AA^\top\| + 1+2\sqrt{\mu}+L(1+\beta).$$ Clearly, $G(t,\cdot)$ is $K(t)$-Lipschitz continuous with $K(t)\in L_{\rm loc}^{1}([t_0,+\infty))$.
		
		Next, for any $t\geq t_0$ and $Z(t)\in\mathbb{R}^n\times\mathbb{R}^n\times\mathbb{R}^m$, we show that there exists $P(t)\in L_{\rm loc}^{1}([t_0,+\infty))$ such that $\|G(t,Z(t))\|\leq P(t)(1+\|Z(t)\|)$. Indeed,
		\begin{eqnarray} \label{Gt}
			\begin{split}
				& \|G(t,Z(t))\|\\
				\leq &~ \left(1 + 2\sqrt{\mu} + \Big(\frac{1}{\sqrt{\mu}} - \beta\Big)\eta(t)\|A\| \right)\|a(t)\| \\
				& + \left(\|A^\top\| + \frac{\beta}{\sqrt{\mu}}\eta(t)\|AA^\top\|\right)\|\lambda(t)\| +\eta(t)\|A\|\|x(t)\| \\
				&  +\Big(1+\frac{\beta\|A\|}{\sqrt{\mu}}\eta(t)\Big)\|\nabla f(x(t)+\beta a(t))\| + \eta(t)\|b\|.
			\end{split}
		\end{eqnarray}
		By the Lipschitz continuity of $\nabla f$, we have $$\|\nabla f(x(t)+\beta a(t))-\nabla f(0)\|\leq L\|x(t)+\beta a(t)\|\leq L\|x(t)\|+L\beta\|a(t)\|.$$ This together with \eqref{Gt} yields
		\begin{eqnarray*}
			\begin{split}
				& \|G(t,Z(t))\|\\
				\leq &~ \left(1 + 2\sqrt{\mu} + \left(\frac{1}{\sqrt{\mu}} - \beta\right)\eta(t)\|A\|+ L\beta \Big(1+\frac{\beta\|A\|}{\sqrt{\mu}}\eta(t)\Big)\right)\|a(t)\| \\
				& + \left(\|A^\top\| + \frac{\beta}{\sqrt{\mu}}\eta(t)\|AA^\top\|\right)\|\lambda(t)\| + \left( \eta(t)\|A\| + L\Big(1+\frac{\beta\|A\|}{\sqrt{\mu}}\eta(t)\Big)\right)\|x(t)\| \\
				& + \eta(t)\|b\| +\Big(1+\frac{\beta\|A\|}{\sqrt{\mu}}\eta(t)\Big)\|\nabla f(0)\| \\
				\leq &~ P(t)(1+\|Z(t)\|),
			\end{split}
		\end{eqnarray*}
		where $P(t):= \left(1+\frac{1}{\sqrt{\mu}}-\beta +\frac{\beta}{\sqrt{\mu}}(L+L\beta + \|\nabla f(0)\|)\right)\eta(t)\|A\| + \|A^\top\| + \frac{\beta}{\sqrt{\mu}}\eta(t)\|AA^\top\| + 1+2\sqrt{\mu}+L(1+\beta)+\|\nabla f(0)\|+ \eta(t)\|b\|$. Clearly, $P(t)\in L_{\rm loc}^{1}([t_0,+\infty))$ and
		$$
		\|G(t,Z(t))\|\leq P(t)(1+\|Z(t)\|),\qquad \forall Z(t)\in\mathbb{R}^n\times\mathbb{R}^n\times\mathbb{R}^m.
		$$
		
		Finally, according to \cite[Proposition 6.2.1]{exist}, for any initial condition $(x(t_0),\dot{x}(t_0),\lambda(t_0))\in\mathbb{R}^n\times\mathbb{R}^n\times\mathbb{R}^m$, System \eqref{system} admits a unique global strong solution. The proof is complete. \qed
		
	\end{proof}

\section*{Funding}
{\small This research  is supported by the Natural Science Foundation of Chongqing (CSTC2025NSCQ-GPX0810 and CSTB2024NSCQ-MSX0651),  the Team Building Project for Graduate Tutors in Chongqing (yds223010). }
\section*{Data availability}
{\small The authors confirm that all data generated or analysed during this study are included in this article.}
\section*{Declaration}
{\small $\mathbf{Conflict ~of~ interest}$ No potential conflict of interest was reported by the authors.}

\bibliographystyle{plain}

\end{document}